\documentclass{amsart}

\usepackage{amssymb}
\usepackage{amsmath,amsthm}
\usepackage{amsthm}
\usepackage{color}
\usepackage{hyperref,url}
\usepackage{graphicx}
\usepackage{subcaption}

\calclayout

\hypersetup{
    colorlinks=true,
    linkcolor=blue,
    filecolor=blue,
    urlcolor=black,
    citecolor=blue,
}

\numberwithin{equation}{section}

\definecolor{plum}{rgb}{.5,0,1}

\DeclareMathOperator{\Arg}{Arg}

\theoremstyle{plain}
\newtheorem{theorem}{Theorem}[section]
\newtheorem{proposition}[theorem]{Proposition}
\newtheorem{corollary}[theorem]{Corollary}
\newtheorem{lemma}[theorem]{Lemma}

\theoremstyle{definition}

\newtheorem{definition}{Definition}[section]

\newtheorem*{remark}{Remark}

\newcommand{\N}{\mathbb{N}}

\newcommand{\R}{\mathbb{R}}
\newcommand{\C}{\mathbb{C}}

\newcommand{\e}{\epsilon}

\newcommand{\supp}{\mathrm{supp}\, }

\newcommand{\mc}{\mathcal}

\DeclareFontFamily{U}{mathx}{\hyphenchar\font45}
\DeclareFontShape{U}{mathx}{m}{n}{
      <5> <6> <7> <8> <9> <10>
      <10.95> <12> <14.4> <17.28> <20.74> <24.88>
      mathx10
      }{}
\DeclareSymbolFont{mathx}{U}{mathx}{m}{n}
\DeclareFontSubstitution{U}{mathx}{m}{n}
\DeclareMathAccent{\widecheck}{0}{mathx}{"71}
\DeclareMathAccent{\wideparen}{0}{mathx}{"75}

\title{Zeros of Dynamical Zeta Functions for Hyperbolic Quadratic Maps}
\date{\today}
\author{Yuqiu Fu}
\address{Department of Mathematics, MIT, Cambridge, MA 02139}
\email{ \href{mailto:yuqiufu@mit.edu}{yuqiufu@mit.edu}}

\begin{document}

\begin{abstract}
  We prove that the dynamical zeta function $Z(s)$ associated to $z^2+c$ with $c<-3.75$ has essential zero-free strips of size $1/2+,$ that is, for every $\e>0,$ there exist only finitely many zeros in the strip $\mathrm{Re}(s)>1/2+\e.$ We also present some numerical plots of zeros of $Z(s)$ based on the method proposed in \cite{jenkinson2002calculating}.
\end{abstract}

\maketitle

\section{Introduction}

We consider the rational map on the complex plane $f(z)=z^2+c.$ We let $\mathcal{J}$ be the Julia set associated to $f.$ For $c<-2,$ it can be shown (for example by using Theorem 2.2 of \cite{urbanski2003measures}) that $f$ is hyperbolic, that is,
\begin{equation}\label{1}
  \inf_{z\in \mathcal{J}}\left\{|(f^{(N)})'(z)|\right\}>1
\end{equation}
for some $N>0.$ From now on we will assume $c<-2.$


The dynamical zeta function associated to the dynamical system $f$ can be defined in the following way \cite{strain2004growth,naud2005expanding}.

We let $\zeta_c=\frac{\sqrt{1-4c}+1}{2}$ be the largest fixed point of $f.$ Under our assumption $c<-2$ we can easily check that $2<\zeta_c<-c.$ We choose $R\in (\zeta_c,-c).$ Then we may verify that $\overline{f^{-1}(D(0,R))}\subset D(0,R).$ Here $D(0,R)$ is the disc of radius $R$ on the complex plane. Let $$g_+(z)=\sqrt{z-c},\quad g_-(z)=-\sqrt{z-c}$$
be the two inverse branches of $f$ on $D(0,R).$ Here the square root function is defined on $\C\setminus (-\infty,0]$ and is real on the real axis.
Then we have $\overline{g_{i}(D(0,R))}\subset D(0,R)$ for $i\in \{+,-\}.$ We denote $g_{i}(D(0,R))$ by $D_{i}$ for $i\in \{+,-\}.$ See Figure \ref{fig:Illustration} for an illustration.

\begin{figure}
  \includegraphics[width=105mm]{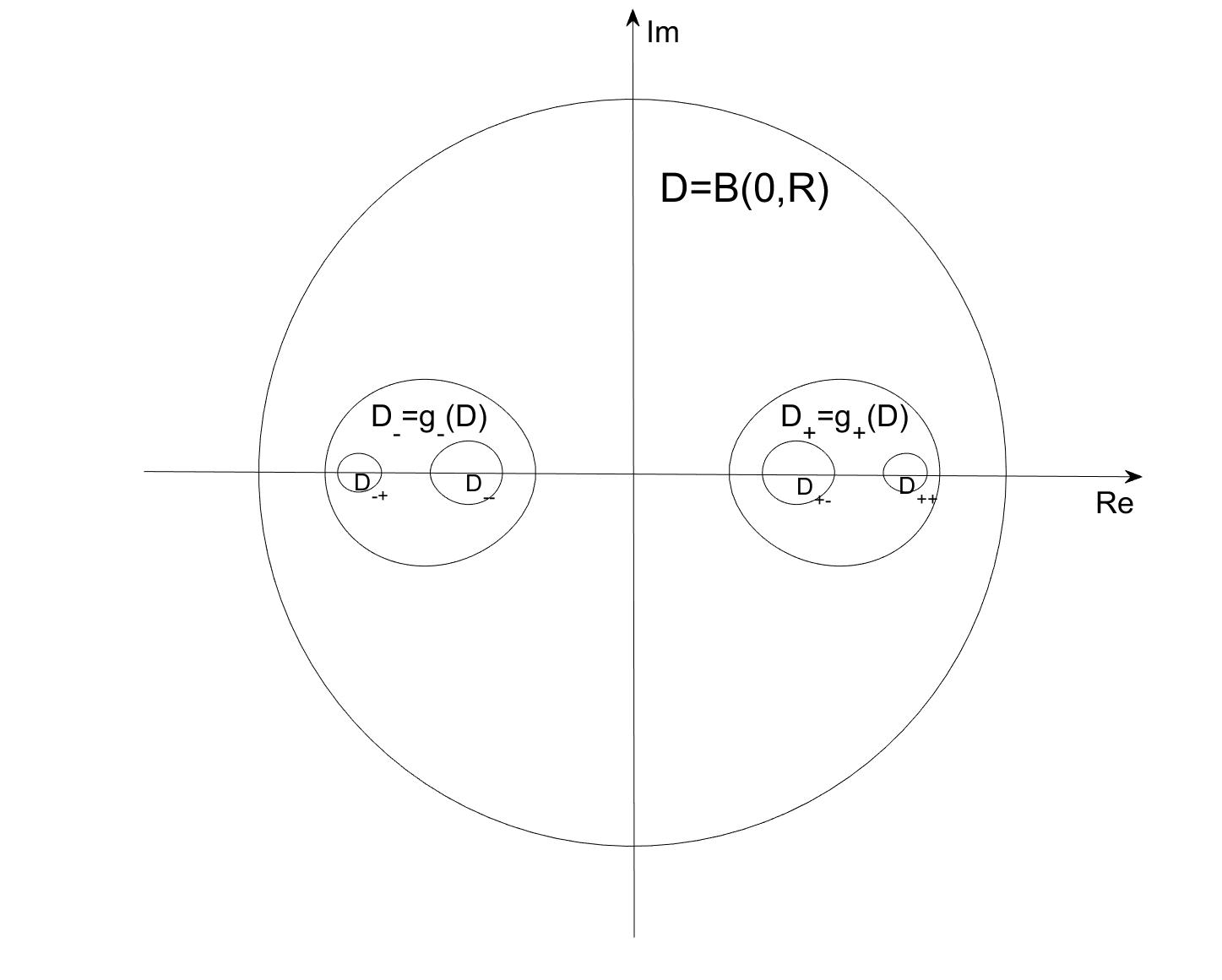}
  \caption{Illustration of $D,D_+,D_-,$ and $D_{ij}:=g_i(g_j(D)).$}
  \label{fig:Illustration}
\end{figure}

If we let $I_+=[\sqrt{-\zeta_c-c},\zeta_c],$ $I_-=[-\zeta_c, -\sqrt{-\zeta_c -c }],$ then it can be easily checked that
$$I_i\subset g_i(D(0,R)),\quad g_i(I_j)\subset I_i.$$
We have $\overline{g_{j}(D_i)}\subset D_j.$ Also, when $|z|>\zeta_c,$ $f^{kN}(z)$ goes to infinity as $k\rightarrow \infty$ and therefore the Julia set of $f$ is given by
$$\mathcal{J}=\bigcap_{k\geq 1}f^{-kN}(D(0,R))=\bigcap_{k\geq 1}f^{-kN}(I_1\cup I_2).$$
In particular we have $\mathcal{J}\subset I_1\cup I_2.$

Define $\mathcal{H}$ to be the Bergman space
$$\mathcal{H}:=\left\{ u\in L^2(D_+\cup D_-): u \text{ is holomorphic on $D_+\cup D_+$ } \right\}.$$
We define the Ruelle transfer operator $\mathcal{L}_s:\mathcal{H}\rightarrow \mathcal{H}$ by
$$\mathcal{L}_su(z)=\sum_{i\in \{+,-\}}[g_i'(z)]^s u(g_i(z)), \qquad z\in D_j.$$
Here $[g_i']$ is the analytic continuation to $D_1\cup D_2$ of $|g_i'|$ defined on $\R \cap (D_1\cup D_2).$ $\mathcal{L}$ is a trace class operator, and therefore we could define the dynamical zeta function by
$$Z(s)=\det(I-\mathcal{L}_s).$$
If $s\in \C$ is a zero of the dynamical zeta function $Z(s),$ then there exists a nonzero $u\in \mathcal{H}$ such that $\mathcal{L}_su=u.$  Our main theorem will be the following.

\begin{theorem}\label{maintheorem}
  Suppose $c<-3.75.$ For every $\e>0,$ there exists $M>0$ such that for $s\in \C$ with $\mathrm{Re}(s)>1/2+\e,$ $|\mathrm{Im}(s)|>M,$ and for every $u\in \mathcal{H},$
  $$\mathcal{L}_su=u \quad \Rightarrow \quad u=0.$$
\end{theorem}

\begin{corollary}\label{maincoro}
  Suppose $c<-3.75.$ Then for every $\e>0,$ $Z(s)$ has only finitely many zeros in the half plane $\mathrm{Re}(s)>1/2+\e.$
\end{corollary}

Let $\delta$ denote the Hausdorff dimension of the Julia set $\mathcal{J}$ associated to $f.$
It has been proved that $Z(s)$ has no zeros in the strip $\mathrm{Re}(s)>\delta-\e$ other than a simple zero at $s=\delta,$ for some $\e>0$ -- see \cite{naud2005expanding}. Therefore our result is only meaningful when $\delta>1/2.$ See Figure \ref{fig:dimension} for a plot of $\delta$ with respect to $c.$ The graph shows that our result provides an improvement on the essential zero-free strips of $Z(s)$ for $c$ in, for example, the range $(-4.6,-3.75).$

\begin{figure}
  \includegraphics[width=95mm]{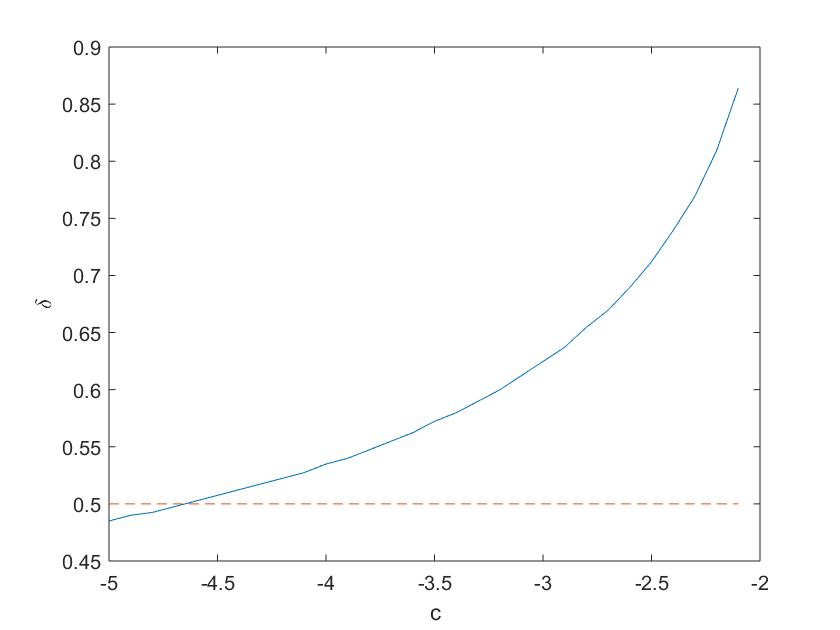}
  \caption{The solid curve in the plot shows how the Hausdorff dimension $\delta$ of the Julia set varies as $c$ varies. The dashed line is the line $\delta=0.5.$ For $c>-4.6,$ we have $\delta>0.5.$}
  \label{fig:dimension}
\end{figure}

We will prove some properties of iterates of $f(z)=z^2+c$ in Section \ref{propertysection}, introduce a refined transfer operator in Section \ref{refinedsec}, and obtain some a priori bounds in Section \ref{priorisec}. Finally we will complete the proof of Theorem \ref{maintheorem} in Section \ref{proofsection}. We will also present some numerical results in Section \ref{numericsection}.

\medbreak
\noindent
{\bf Notation. } We let $\mathcal{W}$ be the set of words generated by alphabets (or letters) $+,-.$ That is, if we let $\mathcal{W}_n$ be the set of words of length $n$
$$\mathcal{W}_n:=\left\{a_1\cdots a_n: a_i\in \{+,-\}\right\},$$
then $\mathcal{W}=\bigcup_{n\geq 1}\mathcal{W}_n.$ Let $|w|$ denote the length of the word $w.$ For $w=w_1\cdots w_n\in \mathcal{W}_n,$ we let $g_w$ be the function
\begin{equation}\label{gw}
  g_w=g_{w_n}\circ \cdots \circ g_{w_1}.
\end{equation}
For $m\leq n,$ we define $w_{<m}$ to be the word $w_1\cdots w_{m-1}.$ We can define $w_{m_1\leq \cdot \leq m}$ and other similar expressions in the obvious way.

$\hat{f}$ will denote the Fourier transform of $f:$
$$\hat{f}(\xi)=\int_{\R^n} f(x)e^{-ix\cdot \xi}dx$$
and $\check{f}$ will denote the inverse Fourier transform such that $(\hat{f})^{\check{}}=f.$ We let $\mathcal{F}_h(f)$ be the semiclassical Fourier transform of $h,$ which is defined by $\mathcal{F}_h(f)(\xi)=\hat{f}(\xi/h).$

We let $A\lesssim B$ denote the statement that there exists a constant $C>0$ such that $A\leq CB.$ The constant may depend on our dynamical system we set up in this section and various parameters chosen in sections below, which only depend on the dynamical system and $\e$ (as in Theorem \ref{maintheorem}). In fact we can regard those constants as only depending on $c$ and $\e.$  We call such constants admissible. We let $B\gtrsim A$ be the statement $A\lesssim B,$ and let $A\sim B$ be the statement that $A\lesssim B$ and $B\lesssim A.
$ We will also use the notation $A\lesssim_j B$ or $A = \mathcal{O}_j (B)$ if there exist an admissible constant $C>0$ and a constant $C_j$ which depends on $j$ such that $A\leq CC_j B.$ Similarly, we define $A\gtrsim_j B$ and $A\sim_j B.$

\section{Properties of the Dynamical System $f(z)=z^2+c$}\label{propertysection}
We recall that the Julia set $\mathcal{J}$ is preserved under backward and forward iteration \cite{carleson2013complex}:
$$f(\mathcal{J})=\mathcal{J},\quad f^{-1}(\mathcal{J})=\mathcal{J}.$$
Due to hyperbolicity (\ref{1}) we know that
$$\sup_{|w|=N} \sup_{z\in \mathcal{J}} |g_w'(z)|<1.$$

Since $\sup _{|w|=N} |g_{N}''(z)|$ is bounded on $\overline{D(0,R)},$ there exists $\delta>0$ such that
\begin{equation}\label{5}
  \sup_{|w|=N} \sup_{z\in \mathcal{J}(\delta)}|g_{N}'(z)|= \gamma<1.
\end{equation}
Here $\mathcal{J}(\delta)=\mathcal{J}+\overline{D(0,\delta)}$ is the closed $\delta$ neighborhood of $\mathcal{J}$ on the complex plane. We denote $\mathcal{J}(\delta)\cap \R$ by $I.$ By choosing $\delta$ sufficiently small, we can assure that $\mathcal{J}(\delta)\subset D(0,R)\setminus \{0\},$ and $I$ is a finite union of disjoint closed intervals. The latter can be found in \cite[Proposition 3]{strain2004growth}. We write $I=\bigcup_{\alpha=1}^K I^\alpha,$ where $I^\alpha$ are closed intervals.

Because $g_w$ with $|w|=N$ is strictly contracting on $\mathcal{J}(\delta)$ and $g_w(\mathcal{J})\subset \mathcal{J},$ we conclude that when $|w|=N,$ $g_w:\mathcal{J}(\delta)\rightarrow \mathcal{J}(\delta)$ and in particular
\begin{equation}\label{2}
  g_w:I\rightarrow I,\quad \text{if } |w|=N.
\end{equation}
Note that since $\sup_{i\in \{+,-\}}\sup_{z\in \overline{D(0,R)}}|g_i'(z)|\lesssim 1,$ we actually have
$$|g_{w}'(z)|\lesssim \gamma^{|w|/N} \text{ for every }z\in \mathcal{J}(\delta).$$
Consequently if $n$ is large enough, we always have
\begin{equation}\label{141}
  g_{w}:\mathcal{J}(\delta)\rightarrow \mathcal{J}(\delta),\quad
g_{w}(I)\subset I, \text{ if }|w|=n.
\end{equation}
We will always assume that $|w|$ is large enough such that the above mapping invariance holds.

Denote $g_w(I^\alpha)$ by $I_w^\alpha,$ and let $I_n=\bigcup_{|w|=n, \alpha} I_w^{\alpha}=\bigcup_{|w|=n}g_w(I).$ Then we have the following proposition.
\begin{proposition}\label{disjointprop}
  If $n$ is large enough (larger than some admissible constant) then $\{I_w^\alpha: \alpha, |w|=n \}$ is a collection of disjoint closed intervals contained in $I.$
\end{proposition}
\begin{proof}
  Since $I^\alpha$ are closed intervals and $g_w$ are continuous, $I_w^\alpha$ are closed intervals. (\ref{141}) implies that $I_w^\alpha\subset I$ if $|w|$ is large enough. To show they are disjoint, we suppose $g_{w}(z_1)=g_{w'}(z_2)$ for some $z_1\in I^{\alpha_1},$ $z_2\in I^{\alpha
  _2}$ and $|w|=|w'|=n.$ If $w=w',$ then since $g_{+},$ $g_{-}$ are injective on $I,$ we conclude that $z_1=z_2.$ If $w\neq w',$ we let $j$ be the largest integer such that the $j$th letter of $w$ differs from that of $w'.$ As before by injectivity of $g_{\pm}$ on $I,$ we conclude $g_{w_{\leq j}}(z_1)=g_{w'_{\leq j}}(z_2).$ However, this cannot happen because $g_+(I)$ and $g_-(I)$ are disjoint.
\end{proof}

We remark the following distortion theorem, which can be found in for instance \cite{carleson2013complex}.
\begin{theorem}{\label{koebe}}
  Suppose $F$ is holomorphic and injective on the disc $D(0,\delta_0)$ for some $\delta_0>0,$  and $|F'(0)|=M.$ Then
  $$M\frac{1-|z|/\delta_0}{(1+|z|/\delta_0)^3} \leq |F'(z)|\leq
  M\frac{1+|z|/\delta_0}{(1-|z|/\delta_0)^3}.$$
\end{theorem}
Note that there exists $R'<R$ such that $\mathcal{J}(\delta)\subset D(0,R')\subset D(0,R).$ Since $g_+,$ $g_-$ are injective on $D,$ $g_w$ is injective on $D$ for every word $w.$ Therefore we obtain the following corollary.
\begin{corollary}{\label{distortion}}
  There exists an admissible constant $C>0$ such that for every word $w,$
  $$C^{-1}|g_w'(z_1)|\leq |g_{w}'(z_2)|\leq C|g_w'(z_1)|$$
  for every $z_1,z_2\in D(0,R').$
\end{corollary}
\begin{proof}
  By Theorem \ref{koebe} there exists an admissible constant $C_1>0$ such that for every word $w,$
  $$C^{-1}_1|g_w'(z)|\leq |g_{w}'(0)|\leq C_1|g_w'(z)|$$
  for every $z\in D(0,R').$
  Taking $C=C_1^2$ completes the proof.
\end{proof}

We let $\gamma_w=\sup_{z\in \mathcal{J}(\delta)}|g_w'(z)|.$ Then Corollary \ref{distortion} implies 
\[ g_{w}'(z) \sim \gamma_w \text{ for every $z\in \mathcal{J}(\delta).$} \]
Also, \eqref{141} implies
$$\gamma_w\lesssim  \gamma^{|w|/N}.$$

Choose $R'' < R'$ such that $\mc{J}(\delta) \subset D(0, R'').$ The above corollary combined with the Cauchy integral formula implies the following.
\begin{proposition}{\label{secondder}}
  For every integer $k \geq 2,$ every $w,$ and every $z\in D(0, R''),$ we have
  $$|g^{(k)}_w(z)|\lesssim_k |g'_w(z)|.$$
\end{proposition}
\begin{proof}
  Choose $R''' \in [R'', R'].$
  By the Cauchy integral formula we have, for every $z \in D(0,R'')$
  \[ g_w^{(k)}(z) = \frac{(k-1)!}{2i\pi} \int_{\partial D(0, R''')} \frac{g_w'(x)}{(x-z)^{k}} dx. \]
  Since $d(z, \partial D(0,R''')) \geq R''' - R'' \gtrsim 1,$ we have
  \[|g_w^{(k)}(z)| \lesssim_k \sup_{x \in \partial D(0,R''')} |g_w'(x)| \sim |g_w'(z)|, \]
  where the last step is due to Corollary \ref{distortion}.
\end{proof}

Besides the above observations, in this section we are mainly interested in the phase function, defined for every $w\in \mathcal{W},$
\begin{equation}\label{phasefn}
  \phi_w(z)=\log |g_w'(z)| \text{ for } z\in \mc{J}(\delta).
\end{equation}

Let $w,v$ be two words and let $m$ be the largest integer such that $v_{<m}=w_{<m}.\footnote{We use the convention that if $w_{\leq 1}\neq v_{\leq 1},$ then $m=1.$ If $|w|>|v|$ and $w_{\leq |v|}=v,$ then we let $m=|v|+1.$}$ If $m\leq \min \{|w|,|v|\}-1,$ then we write $w\nsim v.$ Obviously $w \nsim v$ if and only if $v \nsim w.$
The following proposition establishes separation of the phase functions $\partial_x \phi_w (x).$
\begin{proposition}\label{phaseseppre}
  Suppose $w\nsim v.$ If $c<-2-\sqrt{5}$ then by choosing $\delta$ small enough depending only on $c,$ we have for every $x\in I$
  $$|\partial_x \phi_w(x)-\partial_x \phi_v(x)|\gtrsim |g'_{w_{\leq m+1}}(x)|.$$
  Here $m$ is the largest integer such that $v_{<m}=w_{<m}.$
\end{proposition}
\begin{proof}
  We let $n=|w|$ and $k=|v|.$
  By the inverse function theorem we can compute
  $$g'_{\pm}(x)=\frac{1}{2g_{\pm}(x)},\quad g_{\pm}''(x)=\frac{-1}{4(g_{\pm}(x))^3}.$$
  If we let $x_j(x)=g_{w_{\leq j}}(x),$ then by the chain rule we have
  $$g_w'(x)=g_{w_n}'(x_{n-1})\cdots g_{w_1}'(x_0)$$
  where $x_0:=x.$
  So by the definition of the phase function (\ref{phasefn}) we have
  $$\partial_x\phi_w(x)=\partial_x (\log|g'_{w_n}(x_{n-1})|)+\cdots +\partial_x ( \log|g_{w_1}'(x_0)|)=\frac{g_{w_n}''(x_{n-1})x_{n-1}'(x)}{g_{w_n}'(x_{n-1})}
  +\cdots + \frac{g_{w_1}''(x_0)}{g_{w_1}'(x_0)}.$$
  Since $x_j'(x)=g_{w_j}'(x_{j-1})\cdots g_{w_1}'(x_0),$ we conclude the formula
  \begin{align}\label{phaseformula}
    \partial_x \phi_w(x) & = \frac{g_{w_n}''(x_{n-1})g_{w_{n-1}}'(x_{n-2})\cdots g'_{w_1}(x_0)}{g'_{w_n}(x_{n-1})}+\cdots +\frac{g_{w_1}''(x_0)}{g_{w_1}'(x_0)} \\ \nonumber
    & =-\frac{1}{2^nx_{1}\cdots x_{n-1}x_{n}^2}-\cdots -\frac{1}{2x_1^2}  \\ \nonumber
    & =s_n+\cdots +s_1, \nonumber
  \end{align}
  where by definition $s_i=-{1}/(2^ix_1\cdots x_i^2).$ Similarly we can write
  \begin{align*}
    \partial_x \phi_v(x) & =-\frac{1}{2^ky_{1}\cdots y_{k-1}y_{k}^2}-\cdots -\frac{1}{2y_1^2}  \\
    & =t_k+\cdots +t_1,
  \end{align*}
  where $y_j=g_{v_{\leq j}}(x)$ and $t_i=-{1}/(2^iy_1\cdots y_i^2).$

  Since $m$ is the largest integer such that $w_{<m}=v_{<m},$ we have $x_i=y_i$ for $1\leq i< m,$ and $x_{m}=-y_{m}.$ Therefore $s_i=t_i$ for $1\leq i\leq m,$ and $s_{m+1},$ $t_{m+1}$ differ in sign. We write
  \begin{align}\label{23}
    \partial_x \phi_w(x)-\partial_x \phi_v (x) &  =(s_{m+1}+\cdots +s_{n})-(t_{m+1}+\cdots +t_{k}) \\
     & =(s_{m+1}-t_{m+1})+(s_{m+2}+\cdots + s_n)-(t_{m+2}+\cdots +t_{k}). \nonumber
  \end{align}
  We claim that if $\delta$ is chosen small enough depending on $c,$ then
  $$|s_{m+2}+\cdots +s_n|<\beta |s_{m+1}|,\quad |t_{m+2}+\cdots +t_k|<\beta |t_{m+1}|,$$
  where $\beta\in (0,1)$ is some admissible constant.
  We group consecutive terms with alternating signs in the sequence $s_{m+2},\ldots, s_n,$ that is,
  $$s_{m+2}+\cdots + s_n=s'_1+\cdots+s'_{n'},$$
  where $s'_j=s_{i_j}+\cdots +s_{i_{j+1}-1}$ is a sum of consecutive terms, $s_{i_j+l}$ and $s_{i_j+l+1}$ differ in sign for $0\leq l\leq i_{j+1}-i_j-2,$ and $s_{i_j-1}$ and $s_{i_j}$ have the same sign.

  Note that
  $$\frac{s_{i+1}}{s_i}=\frac{2^i x_1\cdots x_i^2}{2^{i+1}x_1\cdots x_i x_{i+1}^2}=\frac{x_i}{2(x_i-c)}.$$
  Our assumption $c<-2-\sqrt{5}$ in particular implies that $c<-3.75.$ We observe that when $c<-3.75$
  \begin{equation}\label{121}
    \sup_{x\in I_+\cup I_-} \left| \frac{x}{2(x-c)} \right|<\theta<1.
  \end{equation}
  Here $\theta$ can be any number larger than $\frac{\zeta_c}{2(-\zeta_c-c)}$ and less than $1.$
  Recalling that $\mathcal{J} \subset I_+\cup I_-$ and $I=\mathcal{J}(\delta)\cap I,$ we therefore conclude that if we choose $\delta$ sufficiently small depending on $c,$ then we have
  $$\sup_{x\in I} \left| \frac{x}{2(x-c)} \right|\leq \theta <1.$$
  Therefore $|s_{i+1}/s_{i}|\leq \theta<1.$
  Since $s'_j$ is a sum of alternating terms with decreasing absolute values, we must have
  \begin{equation}\label{21}
    s_j' \text{ has the same sign as } s_{i_j} \text{ and } |s_j'|\leq |s_{i_j}|.
  \end{equation}

  We also observe that as $I_+=[\sqrt{-\zeta_c-c},\zeta_c],$ the following inequality holds
  $$\sup_{x\in I_+} \left| \frac{x}{2(x-c)} \right|\leq \left| \frac{\zeta_c}{2(\zeta_c-c)} \right|<\eta<1/4.$$
  Here $\eta$ can be any number larger than $\frac{\zeta_c}{2(\zeta_c-c)}$ and less than $1/4.$
  By definition $s_{i_{j}-1}$ and $s_{i_j}$ have the same sign, which implies that $x_{i_{j}-1}>0$ and therefore $x_{i_{j}-1}\in I_+.$ Again if $\delta$ is sufficiently small depending on $c$ then we have
  \begin{equation}\label{22}
    \left| \frac{s_{i_{j}}}{s_{i_{j}-1}} \right|\leq \eta.
  \end{equation}
  Hence combining (\ref{21}) and (\ref{22}) and using the triangle inequality we conclude that
  $$|s_{m+2}+\cdots +s_n|\leq \sum_{j}|s_j'|\leq \sum_{j=0}^{\infty} \eta^j |s'_1|\leq \frac{1}{1-\eta}\theta |s_{m+1}|.$$
  Under our assumption that $c<-2-\sqrt{5},$ we can choose $\eta, \theta$ depending on $c$ and then $\delta$ sufficiently small depending on $\eta, \theta, c$ such that
  $$\frac{1}{1-\eta}\theta<1.$$
  In fact we only need to guarantee that
  $$\frac{\zeta_c/(2(-\zeta_c-c))}{1-\zeta_c/(2(\zeta_c-c))}<1,$$
  which under our general assumption $c<-2$ is equivalent to $c<-2-\sqrt{5}.$
  Therefore there exists an admissible constant $\beta\in (0,1)$ such that
  $$|s_{m+2}+\cdots +s_n|<\beta |s_{m+1}|.$$
  Similarly we have
  $$|t_{m+2}+\cdots +t_k|<\beta |t_{m+1}|.$$
  Therefore the above two estimates, (\ref{23})  together with the triangle inequality imply that
  $$|\partial_x \phi_w(x)-\partial_x \phi_v(x)|\geq (1-\beta)(|s_{m+1}|+|t_{m+1}|).$$
  Recalling that
  $$s_{m+1}=-1/(2^{m+1}x_1\cdots x_mx_{m+1}^2)=-\frac{g_{w_{\leq m+1}}'(x)}{2x_{m+1}},$$
  we have
  $|s_{m+1}|\gtrsim |g'_{w_{\leq m+1}}(x)|.$ Similarly we have
  $|t_{m+1}|\gtrsim |g_{v_{\leq m+1}}'(x)|.$
  We therefore conclude that
  \begin{equation}\label{91}
    |\partial_x \phi_w(x)-\partial_x \phi_v(x)| \gtrsim |g'_{w\leq m+1}(x)|+|g'_{v\leq m+1}(x)| \geq  |g'_{w\leq m+1}(x)|.
  \end{equation}
\end{proof}
We remark that by definition of $m,$ $|g'_{w\leq m+1}(x)| \sim |g'_{v\leq m+1}(x)|$ for $x\in I.$

In fact we can do better with a little more effort.
\begin{proposition}\label{phasesep}
  Suppose $w\nsim v.$ If $c<-3.75$ then by choosing $\delta$ small enough depending only on $c,$ we have for every $x\in I$
  $$|\partial_x \phi_w(x)-\partial_x \phi_v(x)|\gtrsim |g'_{w_{\leq m+1}}(x)|.$$
  Here $m$ is the largest integer such that $v_{<m}=w_{<m}.$
\end{proposition}
\begin{proof}
  Still we let $n=|w|,$ $k=|v|$ and let $m$ be the largest integer such that $w_{<m}=v_{<m}.$ Under the assumption $w\nsim v$ we have $m\leq \min \{n,k\}-1.$

  If we examine the proof of Proposition \ref{phaseseppre}, we see that under the assumption $c<-3.75,$ the range of $c$ should be determined by the case
  $$x_{m+1}\in I_-, \quad x_{m+i}\in I_+ \text{ for every } i\geq 2, \quad y_{m+1}\in I_-, \quad y_{m+i}\in I_+ \text{ for every } i\geq 2.$$
  To be more precise, for a fixed $x \in I,$ if the word $w$ makes this case happen, the ratio between the alternating sum $\sum_{j}s'_{j}$ and $s_{m+1}$ is the smallest in absolute value among all possible words $w$ with the same length, and this smallest ratio decreases as the length of the words increases when $m$ is fixed.

  To check the claim in the previous paragraph, we fix $x\in I$ and $m.$ Let $i_0$ denote the first positive integer such that $x_{m+i_0} \in I_{-}.$ If such an $i_0$ does not exist, then $x_{m+i} \in I_+$ for every $i\geq 1,$ which implies $\sum_{j} s_{j}'$ has the same sign as $s_{m+1},$ and hence the ratio $(\sum_{j} s_{j}') / s_{m+1} \geq 0.$ Now suppose such an $i_0$ does exist. We first observe that
  \begin{equation}\label{nn2}
    \inf_{x\in I_-} \left|\frac{x}{2(x-c)}\right| > \frac{1}{2} > \sup_{x\in I_+} \left|\frac{x}{2(x-c)}\right|.
  \end{equation}
  Let $\tilde{w}$ denote the word with the same length as $w,$ with $\tilde{w}_{{m}+1} = -, \tilde{w}_{{m}+i} = +$ for $i\geq 2.$ We write $\tilde{x}_{i},$ $\tilde{s}_{m+i}$ and $\tilde{s}'_{j}$ for the corresponding expression defined as $x_i,$ $s_{m+i}$ and $s'_{j}$ for $w.$ If $i_0 =1,$ then \eqref{nn2} implies that $\tilde{s}_{m+1} = s_{m+1},$  $|\tilde{s}_{m+i}| \geq |s_{m+i}|$ for every $i\geq 2,$ and therefore $(\sum_{j} \tilde{s}'_j)/\tilde{s}_{m+i_0} \leq (\sum_j s'_j)/s_{m+i_0}.$
  If $i_0>1,$ then $s_{m+i_0}$ has the same sign as $s_{m+1},$ and $s'_{1}/s_{m+i_0}\geq 0,$ and therefore $(\sum_{j\geq 1} s_j')/s_{m+i_0}\geq (\sum_{j>1} s_j')/s_{m+i_0}.$ The previous analysis implies for the ratio $(\sum_{j>1} s_j')/s_{m+i_0}$ to be the smallest among all possible words with the same length and the same first $m+i_0-1$ letters, we must have $w_{m+i_0+i} = +$ for every $i\geq 1.$ The argument below will show that then under the condition $c<-3.75,$  $s_{m+i_0} + \sum_{j>1} s_j'$ has the same sign as $s_{m+i_0}.$ Since $s_{m+i_0}$ has the same sign as $s_{m+1},$ we have $(\sum_j s'_j)/s_{m+1} \geq (\sum_{j>1} s'_{j})/s_{m+1} \geq 0.$ On the other hand, $ (\sum_{j} \tilde{s}'_j)/\tilde{s}_{m+i_0} \leq 0.$ This completes the proof of the claim in the previous paragraph.


  Now we return to the proof of the proposition. By our assumption $x_m$ and $y_m$ differ in sign. Without loss of generality we may assume that $x_m\in I_-$ and $y_{m} \in I_+.$
  As we have more information about where $x_m,$ $x_{m+1},$ $x_{m+2}$ lie, when estimating $|s_{m+2}/s_{m+1}|,$ $|s_{m+3}/s_{m+2}|,$ $|t_{m+2}/t_{m+1}|,$ $|t_{m+3}/t_{m+2}|,$ we can do better than simply bounding them by $\eta$ or $\theta.$ In fact, as
  $$g_+(I_-)=[\sqrt{-\zeta_c-c},g_+(-\sqrt{-\zeta_c-c})], \quad g_-(I_-)=[g_-(-\sqrt{-\zeta_c-c}),-\sqrt{-\zeta_c-c}],$$
  we have
  $$\left| \frac{s_{m+2}}{s_{m+1}} \right|\leq
  \frac{-g_-(-\sqrt{-\zeta_c-c})}{2(g_-(-\sqrt{-\zeta_c-c})-c)}+o(1),$$
  $$\left| \frac{s_{m+3}}{s_{m+2}} \right| \leq
  \frac{g_+(-\sqrt{-\zeta_c-c})}{2(g_+(-\sqrt{-\zeta_c-c})-c)}+o(1).$$
  Here $o(1)$ denotes a term which goes to $0$ as $\delta\rightarrow 0.$
  Observe that
  $$g_+(g_-(I_+))=[\sqrt{-\zeta_c-c},g_+(g_-(-\sqrt{-\zeta_c-c}))].$$
  So we obtain
  $$\left| \frac{t_{m+3}}{t_{m+2}} \right| \leq
  \frac{g_+(g_-(\sqrt{-\zeta_c-c}))}{2(g_+(g_-(\sqrt{-\zeta_c-c}))-c)}+o(1).$$

  We let $\theta_0=\frac{\zeta_c}{2(-\zeta_c-c)}$ and $\eta_0=\frac{\zeta_c}{2(\zeta_c-c)}.$
  Since $s_m=t_m,$ we also have
  $$\left| \frac{s_{m+1}}{t_{m+1}}\right|\geq \frac{\frac{\sqrt{-\zeta_c-c}}{2(-\sqrt{-\zeta_c-c}-c)}}{\eta_0}+o(1).$$
  Note that in fact
  $$|g'_{w_{\leq m+1}}(x)|\sim |g'_{v_{\leq m+1}}(x)|\sim \max \{|g'_{w_{\leq m+1}}(x)|, |g'_{v_{\leq m+1}}(x)|\} $$
  as $g'_{w_{< m}}(x)=g'_{w_{< m}}(x).$
  So if we have
  $$|s_{m+1}|-|s_{m+2}+\cdots|+|t_{m+1}|-|t_{m+2} +\cdots |\geq \beta \min \{ |s_{m+1}|, |t_{m+1}| \}$$
  for some admissible constant $\beta>0,$ then $|\partial_x \phi_w(x)-\partial_x \phi_v(x)|\gtrsim |g_{w_{\leq m+1}}'(x)|.$
  In our case $|t_{m+1}|\leq |s_{m+1}|$ (when $\delta$ is sufficiently small depending on $c$) since when $c<-3.75,$ $\frac{\frac{\sqrt{-\zeta_c-c}}{2(-\sqrt{-\zeta_c-c}-c)}}{\eta_0}>1.$
  Therefore $c$ can be any real number less than $-3.75$ such that
  \begin{multline}\label{152}
    \left(\theta_0 \left( 1+
  \frac{g_+(g_-(-\lambda))}{2(g_+(g_-(-\lambda))-c)}\frac{1}{1-\eta_0}\right)-1\right) \\
    +
  \frac{\frac{-\lambda}{2(\lambda-c)}}{\eta_0}
  \left(
  \frac{-g_-(\lambda)}{2(g_-(\lambda)-c)}\left(
  1+\frac{g_+(\lambda)}{2(g_+(\lambda)-c)}\frac{1}{1-\eta_0} \right)-1\right)<0,
  \end{multline}
  where $\lambda=-\sqrt{-\zeta_c-c}.$ We can check that (\ref{152}) always holds when $c<-3.75.$
\end{proof}

We also have the following observation on the separation of $\phi''_w.$
\begin{proposition}{\label{phasesecondder}}
  Suppose $w\nsim v, c<-2.$ By choosing $\delta$ small enough depending on $c,$ we have for every $x \in \mc{J}(\delta)$
  $$|\phi_w''(x)-\phi''_{v}(x)|\lesssim |g'_{w_{\leq m+1}}(x)|.$$
  Here $m$ is the largest integer such that $v_{<m}=w_{<m}.$
\end{proposition}
\begin{proof}
  We recall that if $|w|=n_1,$ $|v|=n_2,$ then
  $$\phi'_v(x)=-\left(\frac{g'_v(x)}{y_{n_2}}+\cdots +\frac{g'_{v_{\leq 1}}(x)}{y_1}\right).$$
  Differentiating both sides of the above equation we obtain
  \begin{equation}\label{151}
    \phi_v''(x)=-\left( \frac{g_v''(x)}{y_{n_2}}+\frac{-g_v'(x)g_v'(x)}{y_{n_2}^2}+\cdots +\frac{g''_{w_{\leq 1}}(x)}{y_1}+\frac{-g'_{w_{\leq 1}}(x)g'_{w_{\leq 1}}(x)}{y_1^2} \right).
  \end{equation}
  Since $w_{\leq m}=v_{\leq m},$ we have
  \begin{align}\label{72}
    \phi''_{w}(x)-\phi''_{v}(x) & =\left( \frac{g_{v}''(x)}{y_{n_2}}+\frac{-g_{v}'(x)g_{v}'(x)}{y_{n_2}^2}+\cdots +\frac{g''_{v_{\leq m+1}}(x)}{y_{m+1}}+\frac{-g'_{v_{\leq m+1}}(x)g'_{v_{\leq m+1}}(x)}{y_{m+1}^2} \right)    - \\
     & \quad \left( \frac{g_w''(x)}{x_{n_1}}+\frac{-g'_w(x)g'_w(x)}{x_{n_1}^2}+\cdots +\frac{g''_{w_{\leq m+1}}(x)}{x_{m+1}}+\frac{-g'_{w_{\leq {m+1}}}(x)g'_{w_{\leq {m+1}}}(x)}{x_{m+1}^2} \right).  \nonumber
  \end{align}

  We have the identity, for a general word $w$ and $z \in \mathcal{J}(\delta),$
  \begin{equation}\label{71}
    g_v''(z)=g_v'(z) \left( \frac{-g'_v(z)}{z_{n_2}} + \cdots + \frac{-g_{w_{\leq 1}}'(z)}{z_1} \right)
  \end{equation}
  where $z_j:=g_{w_{\leq j}}(z).$ This is because
  \begin{align*}
    g_w''(z) & =\left(\frac{1}{(2z_n)\cdots (2z_1)}\right)' \\
     & = \frac{-2g'_w(z)}{(2z_n)^2(2z_{n-1})\cdots (2z_1)}+\cdots +\frac{-2g_{w_{\leq 1}}'(z)}{(2z_n)\cdots (2z_2) (2z_1)^2} \\
     & = g_w'(z) \left( -\frac{g_w'(z)}{z_n} - \cdots - \frac{g_{w_{\leq 1}}'(z)}{z_1}  \right).
  \end{align*}

  Noting that $\gamma_w\lesssim \gamma^{|w|/N},$ $\sum_{j=1}^{\infty}\gamma^{j/N}\lesssim 1,$ and $\mathcal{J}(\delta)$ is away from $0,$ we conclude by substituting (\ref{71}) into (\ref{72}) that
  $$|\phi_w''(x)-\phi''_{v}(x)|\lesssim |g'_{w_{\leq m+1}}(x)|+ |g'_{v_{\leq m+1}}(x)|\lesssim |g'_{w_{\leq m+1}}(x)|.$$
  The last inequality is due to $|g'_{w\leq m+1}(x)| \sim |g'_{v\leq m+1}(x)|$ for every $x\in I.$
\end{proof}

As a corollary of the above proposition and the Cauchy integral formula, we have the following corollary.

\begin{corollary}\label{kdercor}
  Suppose $w \nsim v,$ $c < -2,$ and $k\geq 2.$  Then by choosing $\delta$ small enough depending on $c,$ we have for every compact set $\tilde{I} \subset \R $ contained in the interior of $I,$
  $$|\phi_w^{(k)}(x)-\phi^{(k)}_{v}(x)| \lesssim_{k, \tilde{I}} |g'_{w_{\leq m+1}}(x)|$$
  for every $x \in \tilde{I}.$
  Here $m$ is the largest integer such that $v_{<m}=w_{<m}.$
\end{corollary}
\begin{proof}
  Fix a compact $\tilde{I}$ contained in the interior of $I.$ There exists $\delta_0>0$ depending on $\tilde{I}$ such that the closure of the $\delta_0$-neighborhood of $\tilde{I}$ is contained in $\mathcal{J}(\delta).$ By the Cauchy integral formula we have for $x\in \tilde{I},$
  \[ \phi_w^{(k)}(x)-\phi^{(k)}_{v}(x) = \frac{(k-2)!}{2\pi i} \int_{\partial D(x, \delta_0)} \frac{\phi_w''(z)-\phi_{v}''(z)}{(z - x)^{k-1}} dz.  \]
  Therefore, by Proposition \ref{phasesecondder} and Corollary \ref{distortion} we conclude
  $$|\phi_w^{(k)}(x)-\phi^{(k)}_{v}(x)| \lesssim_{k, \tilde{I}} |g'_{w_{\leq m+1}}(x)|.$$
\end{proof}

From now on we will always assume that $\delta$ is sufficiently small (depending only on $c$) such that Proposition \ref{phasesep}, \ref{phasesecondder} and Corollary \ref{kdercor} hold.

\section{Refined Transfer Operator}{\label{refinedsec}}
Our strategy of the proof will be deriving a contradiction from the identity $\mathcal{L}_su=u.$ A naive attempt will be applying iterates of $\mathcal{L}_s$ and therefore obtaining $\mathcal{L}_s^nu=u,$ which we hope would necessarily fail if $n$ is large enough. This method indeed works if our dynamical system is simple enough. However, to tackle our problem here we need to introduce a refined transfer operator $\mathcal{L}_{Z,s},$ which
has the property that $\mathcal{L}_s u=u \Rightarrow \mathcal{L}_{Z,s} u = u$ so that in some sense it is a generalization of simple iterates $\mathcal{L}_s^n.$

\begin{definition}
  We say $Z\subset \mathcal{W}$ is a partition if $Z$ is finite and there exists $M>0$ such that for every word $w$ with length greater than $M,$ there exists a unique $v\in Z$ such that $w_{\leq m}=v$ for some $m\geq 1.$
\end{definition}

\begin{remark}
  We observe that the set
  $$Z(\tau):=\{w\in \mathcal{W}: |I_w|< \tau ,\quad |I_{w_{\leq j}}|\geq \tau \text{ for every } 1\leq j\leq |w|-1 \}$$
  is a partition for every $\tau >0.$ The reason why $Z(\tau)$ is a partition is the fact that $g_w$ is eventually contracting in the sense that $\sup_{z\in \mathcal{J}(\delta)}|g_w'(z)|\lesssim \gamma^{|w|/N}.$
\end{remark}

We let $\mathcal{L}_{Z,s}$ be the operator\footnote{Our definition of $\mathcal{L}_{Z,s}$ differs from that in \cite{dyatlov2017fractal} since in \cite{dyatlov2017fractal}, $g_w$ would be $g_{w_1} \circ \cdots \circ g_{w_n}$ instead of $g_{w_n}\circ \cdots \circ g_{w_1}$ as in (\ref{gw}).}
$$\mathcal{L}_{Z,s}u(z)=\sum_{w\in Z}[g'_w(z)]^su(g_w(z)).$$

\begin{proposition}
  Suppose $Z$ is a partition. If $\mathcal{L}_su=u,$ then
  $$\mathcal{L}_{Z,s}u=u.$$
\end{proposition}
\begin{proof}
  We prove by induction on $K:=\sum_{w\in Z} |w|.$

  If $K=2,$ then $Z$ is a partition if and only if $Z=\{+,-\}$ and therefore $$\mathcal{L}_{Z,s}u=\mathcal{L}_{s}u=u.$$

  Suppose our proposition holds for $K\leq K_1$ where $K_1\geq 2.$ Let $Z$ be a partition with $\sum_{w\in Z} |w|=K_1+1$ and let $v=v_1\cdots v_{n+1}$ be a longest word in $Z.$ Since $Z$ is a partition and $v$ is one of the longest words in $Z,$ if we let $v'$ and $v''$ be words $v_1\cdots v_{n} +$ and $v_1\cdots v_n -$ respectively, then $v',v''\in Z.$ Define $Z'$ to be the set
  $Z\setminus \{v',v''\}\cup {v_{\leq n}}.$ Then $Z'$ is a partition with $\sum_{w\in Z'} |w|\leq K.$ So by induction hypothesis,
  $$\mathcal{L}_{s}u \quad \Rightarrow \quad \mathcal{L}_{Z',s}u=u.$$
  Observe that
  $$\mathcal{L}_{Z,s}u(z)-\mathcal{L}_{Z',s}u(z)=(g'_{v'}(z))^s u(g_{v'}(z))
  +(g'_{v''}(z))^s u(g_{v''}(z))-(g'_{v_{\leq n}}(z))^s u(g_{v_{\leq n}}(z)),$$
  but the right hand side of the above equality is zero by the chain rule if $\mathcal{L}_{s}u=u$ which in particular implies that $\mathcal{L}_{s}u(g_{v_{\leq n}}(z))=u(g_{v_{\leq n}}(z)).$
\end{proof}

Due to our distortion estimate Corollary \ref{distortion} we know that $|g'_w(z)|$ are comparable on the region $\mathcal{J}(\delta).$ We therefore obtain the following estimate on $|g_w'|$ for $w\in Z(\tau).$
\begin{lemma}\label{der-tau}
  Suppose $w\in Z(\tau).$ Then
  $$|g'_w(z)|\sim \tau$$
  for every $z\in \mathcal{J}(\delta).$ In particular we have $\gamma_w\sim \tau.$
\end{lemma}
\begin{proof}
  Recall that $I_w^\alpha=g_w(I^\alpha)$ and for fixed $w,$ $I_w^{\alpha}$ are disjoint closed intervals. Since $|g_w'(x)|\sim 1$ when $w\in \{+,-\},$ $x\in I,$  the definition of $Z(\tau)$ and Lagrange's mean value theorem imply that for every $w\in Z(\tau),$ we have $|I_w|\sim \tau.$

  Similarly due to Lagrange's mean value theorem, the above observation implies that there exists $x\in I$ such that $|g'_w(x)|\sim \tau.$ Hence Corollary \ref{distortion} shows that $|g'_w(z)|\sim \tau$ for every $z\in \mathcal{J}(\delta).$
\end{proof}

We also have the following ``almost orthogonality'' property.
\begin{lemma}\label{almostorth}
  For every $w\in Z(\tau),$ we have
  \begin{equation}\label{153}
    \left| \{v\in Z(\tau): |v|\geq |w| \text{ and } w\sim v\} \right| \lesssim 1.
  \end{equation}
  Here $|A|$ denotes the cardinality of the set $A,$ and $w\sim v$ means the negation of $w\nsim v.$
\end{lemma}
\begin{proof}
  Suppose $w,v\in Z,$ $|v|\geq |w|,$ and $w\sim v.$ We let $|w|=n.$ Then we must have $v_{\leq n-1}=w_{\leq n-1}.$ Due to Lemma \ref{der-tau}, $|g_w'(z)|\sim \tau$ for every $z\in \mathcal{J}(\delta).$ Therefore $|g_{v_{\leq n-1}}(z)|=|g_{w_{\leq n-1}}(z)|\sim \tau$ for every $z\in \mathcal{J}(\delta).$ Again due to Lemma \ref{der-tau} we know that $|g_{v}'(z)|\sim \tau$ for every $z\in \mathcal{J}(\delta).$ Since $g_w$ is eventually contracting on $\mathcal{J}(\delta),$ we conclude that $|v|-n\lesssim 1.$ Hence
  $$\left| \{v\in Z(\tau): |v|\geq |w| \text{ and } w\sim v\} \right| \lesssim 1.$$
\end{proof}

We remark that Lemma \ref{almostorth} also shows that when $\tau$ is sufficiently small, for every $z\in I,$
\begin{equation}\label{n60}
  |\{w \in Z(\tau): z\in g_w(I) \}| \lesssim 1,
\end{equation}
since by the proof of Proposition \ref{disjointprop}, $z\in g_w(I)$ and $z\in g_v(I)$ imply $w\sim v.$

For a word $w$ with length $n> N,$ we let $\underline{w}$ be the word $w_{\leq n-N},$ which is obtained by removing the last $N$ letters of $w.$
Since for every $w\in Z(\tau),$ $|\{v \in Z(\tau): \underline{v} = \underline{w} \}| \lesssim 1,$ we have for every $z\in I,$
\begin{equation}\label{n61}
  |\{w \in Z(\tau): z\in g_{\underline{w}}(I) \}| \lesssim 1.
\end{equation}

\section{A Priori Bounds}{\label{priorisec}}
In this section we  fix $\e>0$ and suppose $\mathrm{Re}(s)>1/2+\e$ and $\mathcal{L}_su=u$ for some $u\in \mathcal{H}.$ We will establish some a priori bounds on the function $u.$ Let $h=1/|\mathrm{Im}(s)|.$ Recall that $N$ is defined by (\ref{1}) and depends only on $c.$
\begin{proposition}\label{priori} We have
  \begin{enumerate}
    \item $\sup_{x\in I_N}|u(x)|\lesssim h^{-1/2}\|u\|_{L^2(I)};$
    \item $\|u^{(k)}\|_{L^2(I_N)}\lesssim_k h^{-k}\|u\|_{L^2(I)},$ for every $k \in \N.$
  \end{enumerate}
\end{proposition}

To prove the above two bounds, we will follow the argument in \cite{dyatlov2017fractal}.
We let $D_N$ be the set $\bigcup_{|w|=N} g_w(\mathcal{J}(\delta)).$ Recall that  $I_N=\bigcup_{|w|=N}g_w(I)=D_N\cap \R.$ Recalling that $g_w$ with $|w|=N$ are strictly contracting on $\mathcal{J}(\delta),$ and map $\mathcal{J}$ to $\mathcal{J},$ we conclude that
\begin{equation}\label{31}
  d(D_N, \partial \mathcal{J}(\delta))\geq (1-\gamma)\delta.
\end{equation}
Here $d(A,B)$ is the Euclidean distance between two subsets $A,B$ of $\R^n.$

Without loss of generality we may assume that $h>0$ from now on.
We introduce the following weight function
$$w_K:\C\rightarrow (0,\infty),\quad z\mapsto e^{-K|\mathrm{Im}(z)|/h}.$$

By Lemma 2.3 in \cite{dyatlov2017fractal}, there exists $c\in (0,1]$ such that
\begin{equation}\label{10}
  \sup_{z\in D_N}|w_K(z)u(z)|\leq \left(\sup_{z\in I}|w_K(z)u(z)| \right)^c\left( \sup_{z\in \mathcal{J}(\delta)} |w_K(z)u(z)| \right)^{1-c}.
\end{equation}
Note that $w_K(z)$ equals to $1$ when $z$ is real.

Since $\mathcal{L}_s^Nu=u,$ for every $z\in \mathcal{J}(\delta)$ we have
\begin{align*}
  |w_K(z)u(z)| & = \left|w_K(z)\sum_{|w|=N} g_w'(z)^su(g_w(z))\right|  \\
   & \leq \gamma^{\mathrm{Re}(s)}\sum_{|w|=N} w_K(z)e^{-\Arg(g_w'(z))/h} |u(g_w(z))| \\
   & \lesssim \left( \sum_{|w|=N} \frac{w_K(z)}{w_K(g_w(z))}e^{- \Arg(g_w'(z)/h)} \right) \sup_{z\in D_N} |w_K(z)u(z)| \\
   & \lesssim e^{(-K|\mathrm{Im}(z)|(1-\gamma)-\Arg(g_w'(z)))/h} \sup_{z\in D_N} |w_K(z)u(z)|.
\end{align*}
Observe that since $\Arg(g_w'(z))=0$ when $\mathrm{Im}(z)=0,$ we can choose an admissible $K>10$ sufficiently large such that
$$\sup_{z\in \mathcal{J}(\delta)}{(-K|\mathrm{Im}(z)|(1-\gamma)-arg(g_w'(z)))}\leq 0,$$
which therefore implies that
$$\sup_{z\in \mathcal{J}(\delta)} |w_K(z)u(z)|\lesssim \sup_{z\in D_N}|w_K(z)u(z)|.$$
Combining the above estimate with (\ref{10}) we therefore conclude
\begin{equation}\label{11}
  \sup_{z\in \mathcal{J}(\delta)} |w_K(z)u(z)|\lesssim \sup_{z\in I} |w_K(z)u(z)|=\sup_{z\in I} |u(z)|.
\end{equation}

We let $\chi:\R\rightarrow [0,1]$ be a smooth cutoff function satisfying
$\chi=1$ on a neighborhood of $I_N$ and $\chi$ is supported on $I.$ This is possible because of (\ref{31}), and we can further assume that
$$\sup_{x\in \R}|\partial_x^j\chi (x)|\lesssim \delta^j(1-\gamma)^j.$$
We let $\tilde{\chi}\in C_c^\infty(\mathcal{J}(\delta))$ be an almost analytic continuation of $\chi$ to the complex plane satisfying
$$\tilde{\chi}=\chi \text{ on }\R,\quad \sup_{z\in \C} |\partial_{\overline{z}}\tilde{\chi}(z)|\lesssim_j |\mathrm{Im}(z)|^j \text{ for every } j\geq 0.$$
A construction of such an almost analytic continuation can be found in \cite[Chapter 8]{dimassi1999spectral}.

Let $u_N(x)=\chi(x)u(x)$. We claim that for every $j>0,$
\begin{equation}\label{41}
  |\mathcal{F}_h(u_N)(\xi)|\lesssim_j h^j|\xi|^{-j}\sup_{I}|u| \text{ when } |\xi|\geq 2K.
\end{equation}
By definition
$$\mathcal{F}_h(u_N)(\xi)=\int_\R e^{-ix\xi/h}u(x)\tilde{\chi}(x)dx.$$
We assume that $\xi\geq 2K.$ The case $\xi\leq -2K$ can be similarly treated.
By Green's formula we have
$$\mathcal{F}_h(u_N)(\xi)=-\int_{\mathcal{J}(\delta)\cap \{\mathrm{Im}(z)\leq 0\}} u(z)e^{-iz\xi/h}\partial_{\overline{z}}\tilde{\chi}d\overline{z}dz.$$
Therefore by (\ref{11}), for every $j$ we have
\begin{align*}
  |\mathcal{F}_h(u_N)(\xi)| & \lesssim_j \sup_{z\in I}|u(z)|\int_{\mathcal{J}(\delta)\cap \{\mathrm{Im}(z)\leq 0\}}e^{\xi \mathrm{Im}(z)/h-K\mathrm{Im}(z)/h}(-\mathrm{Im}(z))^jd\overline{z}dz \\
   & \lesssim_j \sup_{z\in I}|u(z)| \int_{0}^{\infty}e^{y(K-\xi)/h}y^j dy\\
   & \lesssim_j \sup_{z\in I}|u(z)| \int_{0}^{\infty} e^{-\xi y/(2h)}y^jdy \\
   & \lesssim_j \sup_{z\in I}|u(z)|h^j \xi^{-j}
\end{align*}
when $\xi \geq 2K.$ Hence (\ref{41}) is proved.

\begin{proof}[Proof of Proposition \ref{priori}]
  To prove (a) we estimate, using (\ref{41}), H\"{o}lder's inequality and Plancherel's theorem that
  \begin{align}{\label{61}}
    \|u_N\|_{L^\infty(\R)}
     & \lesssim\frac{1}{h}\|\mathcal{F}_h(u_N)\|_{L^1(\R)} \\ \nonumber
     & \lesssim\frac{1}{h}(\|\mathcal{F}_h(u_N)\|_{L^1(-2K,2K)}+
     \|\mathcal{F}_h(u_N)\|_{L^1(\R\setminus (-2K,2K))}) \\ \nonumber
     & \lesssim \frac{1}{h}(\|\mathcal{F}_h(u_N)\|_{L^2(-2K,2K)}+\mathcal{O}(h^\infty)\sup_{x\in I}|u(x)|) \\ \nonumber
     & \lesssim \frac{1}{h}(h^{1/2}\|u_N\|_{L^2(\R)}+\mathcal{O}(h^\infty)\sup_{x\in I}|u(x)| ) \\ \nonumber
     & \lesssim h^{-1/2}\|u\|_{L^2(I)}+\mathcal{O}(h^\infty)\sup_{x\in I}|u(x)|. \nonumber
  \end{align}
  Here $\mathcal{O}(h^\infty)$ denotes a nonnegative term $O$ such that for every $j>0$ there exists an admissible constant $C_j$ such that $O\leq C_j h^j.$
  Noting that by $\mathcal{L}_s^N u=u$ we have
  \begin{equation}\label{nn1}
    \sup_{x\in I}|u(x)|\lesssim
  \sup_{x\in I_N} |u(x)|.
  \end{equation}
  Therefore combining this and (\ref{61}) we get
  $$\|u_N\|_{L^\infty(\R)} \lesssim h^{-1/2}\|u\|_{L^2(I)},$$
  when $h$ is smaller than some admissible constant, which we will assume from now on.
  Since $u_N$ agrees with $u$ on $I_N,$ (a) is proved.

  To prove (b) we estimate similarly:
  \begin{align}\label{62}
    \|u^{(k)}\|_{L^2(I_N)} & \lesssim \|(u_N)^{(k)}\|_{L^2(\R)} \\ \nonumber
     & \lesssim \|\xi^k \widehat{u_N}(\xi)\|_{L^2(\R)} \\ \nonumber
     & \lesssim \frac{1}{h^{k+\frac{1}{2}}}\|\xi^k \mathcal{F}_h(u_N)(\xi)\|_{L^2(\R)} \\ \nonumber
     & \lesssim \frac{1}{h^{k + \frac{1}{2}}}\|\mathcal{F}_h(u_N)\|_{L^2(\R)} +\mathcal{O}(h^\infty)\sup_{x\in I}|u(x)| \\ \nonumber
     & \lesssim_k \frac{1}{h^k} \|u\|_{L^2(I)}+\mathcal{O}(h^\infty)\sup_{x\in I}|u(x)|.
  \end{align}
  Due to (a) and \eqref{nn1} we conclude
  $$\|u^{(k)}\|_{L^2(I_N)}\lesssim_k \frac{1}{h^k} \|u\|_{L^2(I)}$$
  if $h$ is smaller than some admissible constant.
\end{proof}

For a word $w$ with length $n> N,$ we let $\underline{w}$ be the word $w_{\leq n-N},$ which is obtained by removing the last $N$ letters of $w.$

Let $\rho\in (0,1)$ be fixed. We observe that if $h$ is small enough (smaller than some admissible constant), then $|w|>N$ for every $w\in Z(h^{\rho}),$ and $I_{\underline{w}}\subset I_N.$ From now on we will always assume so.
\begin{proposition}{\label{priori2}}
For every word $w\in Z(h^\rho)$  we have
  \begin{enumerate}
    \item $\sup_{x\in I_w}|u(x)|\lesssim h^{-1/2}\|u\|_{L^2(I_{\underline{w}})} +\mathcal{O}(h^\infty)\|u\|_{L^2(I)};$
    \item $\|u^{(k)}\|_{L^2(I_w)}\lesssim_k h^{-k}\|u\|_{L^2(I_{\underline{w}})}+\mathcal{O}(h^\infty)\|u\|_{L^2(I)}.$
  \end{enumerate}
\end{proposition}
To prove the above proposition, one can either repeat the argument in proving Proposition \ref{priori}, by choosing a smooth cutoff function $\chi_w \in C_c^\infty(I_{\underline{w}})$ which equals to $1$ on $I_w$ and then estimate $\|\chi_w u\|_{L^\infty(\R)},$ $\|(\chi_w u)^{(k)}\|_{L^2(\R)}.$

Here we give a proof by utilizing the already established inequality (\ref{41}).

\begin{proof}[Proof of Proposition \ref{priori2}]
  We let $\chi_w\in C_c^\infty(I_{\underline{w}})$ be a smooth cut-off function which equals to $1$ on $I_{w}.$ Due to the strict contracting property (\ref{5}), we can choose $\chi_w$ in a way such that
  \begin{equation}\label{101}
    \sup_{|\xi|>1}|h^{-\rho}\widehat{\chi_w}(\xi)|\lesssim_j
  h^{-j\rho}|\xi|^{-j}
  \end{equation}
  and $\|\chi_w\|_{L^1(\R)}\lesssim 1.$

  As we have seen in the proof of Proposition \ref{priori}, to prove (a) and (b) it suffices to show that
  \begin{equation}\label{102}
    |\mathcal{F}_h(\chi_w u_N)(\xi)|\lesssim_{j} h^j|\xi|^{-j}\|u\|_{L^2(I)} \text{ when }|\xi| \geq 4K.
  \end{equation}
  To show the above inequality we observe that
  \begin{align*}
    \mathcal{F}_h(\chi_w u_N)(\xi) & =\frac{1}{2\pi} \widehat{\chi_w} \ast \widehat{u_N}(\xi/h)  \\
     & =\frac{1}{2h\pi} \int_{\R} \mathcal{F}_h(u_N)(\xi-\eta) \widehat{\chi_w}(\eta/h)d\eta \\
     & =A+B,
  \end{align*}
  where
  \begin{gather*}
    A=\int_{\R\setminus (-\xi/2,\xi/2)} \mathcal{F}_h(u_N)(\xi-\eta) \widehat{\chi_w}(\eta/h)d\eta, \\
    B=\int_{(-\xi/2,\xi/2)} \mathcal{F}_h(u_N)(\xi-\eta) \widehat{\chi_w}(\eta/h)d\eta.
  \end{gather*}
  Suppose $|\xi|\geq 3K.$  Due to (\ref{101}), our assumption $\rho<1$ and the fact that $\|\mathcal{F}_h(u_N)\|_{L^\infty}\lesssim \|u_N\|_{L^1}\lesssim \|u\|_{L^2(I)},$ we obtain
  $$|A|\lesssim_j \|u\|_{L^2(I)} |\xi|^{-j}h^j.$$
  Due to \eqref{41}, our assumption $\rho<1$ and the fact that $\|\widehat{\chi_w}\|_{L^\infty}\lesssim \|\chi_w\|_{L^1}\lesssim 1,$ we have
  $$|B|\lesssim_j  \|u\|_{L^2(I)} |\xi|^{-j}h^j.$$
  Hence (\ref{102}) is proved and so is the proposition.
\end{proof}


\section{Proof of Theorem \ref{maintheorem}} \label{proofsection}
Suppose $c<-3.75$ and fix $\e>0.$  Let $u$ be an element of $\mathcal{H}\setminus \{0\}$ satisfying $\mathcal{L}_su=u$ with $\mathrm{Re}(s)>1/2+\e.$

Let $I_0,\tilde{I}$ be compact subsets of $I$ such that $\tilde{I}$ is contained in the interior of $I,$ $I_0$ is contained in the interior of $\tilde{I},$ and $I_0$ contains an open neighborhood of $\mathcal{J}.$ For every $w \in Z(h^\rho),$ let $\chi_w \in C^\infty_c (\R)$ satisfy $\chi_w(x) = 1$ for $x \in g_w(I_0),$ $\supp \chi_w (x) \subset g_w(\tilde{I}),$ and
\begin{equation}\label{n80}
  |\partial^j \chi_w (x)| \lesssim_j h^{-j\rho} \text{ for every } j \in \N.
\end{equation}

As before we let $h=1/\mathrm{Im}(s)$ and assume that $h>0.$
For every $\rho \in (0,1)$ we have
$$\mathcal{L}_{Z(h^\rho),s}u=u.$$
Because of our choice of $\chi_w,$ we have
\[ \sum_{w\in Z(h^\rho)}|g_w'(x)|^s u(g_w(x)) \chi_w (g_w (x)) = u(x) \text{ for } x\in I_0. \]
In particular this implies
\begin{equation}\label{3}
   \|u\|_{L^2(I_0)}^2 \leq  \left\|\sum_{w\in Z(h^\rho)}|g_w'(x)|^s u(g_w(x)) \chi_w (g_w(x)) \right\|_{L^2(\tilde{I})}^2.
\end{equation}
For $w,w' \in Z(h^\rho)$ with $w\nsim w',$ we let $d(w,w')$ be the largest integer such that $w_{<d(w,w')} = w'_{<d(w,w')}.$ Let $\gamma(w,w')$ denote $\gamma_{w_{\leq d(w,w')+1}}.$
We write
\begin{equation}\label{6}
  \left\|\sum_{w\in Z(h^\rho)}|g_w'(x)|^s u(g_w(x))\chi_w (g_w(x))\right\|_{L^2(\tilde{I})}^2 = T + \sum_{j=-M_1}^{M_2} T_j,
\end{equation}
where
\[ T = \sum_{w,w'\in Z(h^\rho),\, w \sim w'} \int_{\tilde{I}} |g_{w}'(x)|^s\overline{|g_{w'}'(x)|^s} u(g_{w}(x))\overline{u(g_{w'}(x))} \chi_w (g_w(x)) \chi_{w'}(g_{w'}(x)) dx, \]
$$T_j= \sum_{\substack{w,w'\in Z(h^\rho),\, w \nsim w'\\ \gamma(w,w') \in [2^j, 2^{j+1})}} \int_{\tilde{I}} |g_{w}'(x)|^s\overline{|g_{w'}'(x)|^s} u(g_{w}(x))\overline{u(g_{w'}(x))} \chi_w (g_w(x)) \chi_{w'}(g_{w'}(x)) dx ,$$
and $M_2 \in \N$ is an admissible constant, $M_1 \in \N$ satisfies $2^{-M_1} \sim h^\rho.$

We recall by Lemma \ref{der-tau} we have
$$|g'_w(x)| \sim \gamma_w \sim h^\rho$$
for every $w\in Z(h^\rho)$ and $x\in I.$ In the following discussion, all words will be assumed to be in $Z(h^\rho).$

\subsection{Estimating $|T|$}
First, we estimate $|T|.$
By H\"{o}lder's inequality, and the change of variable formula,
\begin{align*}
  \||g_w'(x)|^s u(g_w(x))\chi_w (g_w(x))\|_{L^2(\tilde{I})}^2 & \leq \gamma_w^{2\mathrm{Re}(s)-1} \int_{\tilde{I}}
  |g'_w(x)||u(g_w(x))|^2 \chi_w^2(g_w(x)) dx \\
   & \leq h^{\rho (2\mathrm{Re}(s)-1)} \int_{\R} |u(x)|^2 \chi_w^2(x) dx \\
 & \leq h^{\rho (2\mathrm{Re}(s)-1)} \|u\chi_w\|_{L^2(\R)}^2.
\end{align*}
Similarly, we have
$$\||g_w'(x)|^s u(g_w(x)) \chi_w(x) \|_{L^2(\tilde{I})}\||g'_{w'}(x)|^s u(g_{w'}(x)) \chi_{w'}(x)\|_{L^2(\tilde{I})}\lesssim h^{\rho (2\mathrm{Re}(s)-1)}\|u \chi_w \|_{L^2(\R)}\|u \chi_{w'} \|_{L^2(\R)}.$$
Previous estimates combined with Lemma \ref{almostorth} and the Cauchy-Schwartz inequality show that
\begin{equation}\label{n30}
  |T|\lesssim h^{(2\mathrm{Re}(s)-1)\rho} \sum_{w\in Z(h^\rho)} \|u \chi_w \|_{L^2(\R)}^2 \lesssim h^{(2\mathrm{Re}(s)-1)\rho}\|u\|_{L^2(I)}^2,
\end{equation}
where the last inequality is due to \eqref{n60}.

\subsection{Estimating $|T_j|$}
Now we estimate $|T_j|.$
Because of Proposition \ref{phasesep}, for every pair of words $(w,w')$ with $w\nsim w',$ we can write
\begin{align*}
  & \int_{\tilde{I}} |g_{w}'(x)|^s\overline{|g_{w'}'(x)|^s} u(g_{w}(x))\overline{u(g_{w'}(x))} \chi_w (g_w(x)) \chi_{w'}(g_{w'}(x)) dx \\ & \qquad = \int_{\tilde{I}} |g_{w}'(x)|^{\mathrm{Re}(s)}{|g_{w'}'(x)|^{\mathrm{Re}(s)}} u(g_{w}(x))\overline{u(g_{w'}(x))} \chi_w (g_w(x)) \chi_{w'}(g_{w'}(x))
  e^{i(\phi_{w}(x)-\phi_{w'}(x))/h} dx \\
   & \qquad =\int_{\tilde{I}} |g_{w}'(x)|^{\mathrm{Re}(s)}{|g_{w'}'(x)|^{\mathrm{Re}(s)}} u(g_{w}(x))\overline{u(g_{w'}(x))}
   \chi_w (g_w(x)) \chi_{w'}(g_{w'}(x))
   \left(\frac{h}{i(\phi_{w}'(x)-\phi_{w'}'(x))} \partial_x\right)^k \\
   & \qquad \qquad (e^{i(\phi_{w}(x)-\phi_{w'}(x))/h}) dx
\end{align*}
for every $k \in \N.$ Integrating by parts yields
\begin{align*}
  & \int_{\tilde{I}} |g_{w}'(x)|^s\overline{|g_{w'}'(x)|^s} u(g_{w}(x))\overline{u(g_{w'}(x))} \chi_w (g_w(x)) \chi_{w'}(g_{w'}(x)) dx  \\
   &  =\int_{\tilde{I}} \left(-\frac{h}{i(\phi_{w}'(x)-\phi_{w'}'(x))} (\partial_x - \frac{\phi_w''(x) - \phi_{w'}''(x)}{\phi_w'(x) - \phi_{w'}'(x)})\right)^k \bigg( |g_{w}'(x)|^{\mathrm{Re}(s)}{|g_{w'}'(x)|^{\mathrm{Re}(s)}} u(g_{w}(x))\overline{u(g_{w'}(x))}
    \\
   & \quad
   \chi_w (g_w(x)) \chi_{w'}(g_{w'}(x)) \bigg) e^{i(\phi_{w}(x)-\phi_{w'}(x))/h} dx \\
   & =: A_{w,w'}
\end{align*}

Fix $w,w' \in Z(h^\rho)$ with $w\nsim w'.$ To shorten the notation, we denote
$$b(x) :=\frac{1}{\phi'_w(x) - \phi_{w'}'(x)},$$
and
$$G(x):= |g_{w}'(x)|^{\mathrm{Re}(s)}{|g_{w'}'(x)|^{\mathrm{Re}(s)}} u(g_{w}(x))\overline{u(g_{w'}(x))}
   \chi_w (g_w(x)) \chi_{w'}(g_{w'}(x)).$$
Implicitly $b,G$ depend on $w,w'.$
We can then rewrite $A_{w,w'}$ as
\[ A_{w,w'} = \left( \frac{-h}{i} \right)^k \int_{\tilde{I}} e^{i(\phi_{w}(x)-\phi_{w'}(x))/h} ((\partial_x) (b(x)))^k G(x)  dx, \]
where $(\partial_x) (b(x))$ is the operator $f \mapsto \partial_x (b(x)f(x)).$

We write
$$ B_k(x):=  ((\partial_x) (b(x)))^k G(x).$$
So
\[ A_{w,w'} = \left( \frac{-h}{i} \right)^k \int_{\tilde{I}} e^{i(\phi_{w}(x)-\phi_{w'}(x))/h} B_k(x)  dx. \]

\subsubsection{Estimate of $|B_k(x)|$}
Applying the chain rule repeatedly we see that $B_k(x)$ is a sum of $\mathcal{O}_k (1)$ many terms of the form
\begin{equation}\label{n1}
  (\partial_x^k b)^{j_{k}} \cdots (\partial_x b)^{j_{1}} b^{j_0} \partial_x^{J} G(x),
\end{equation}
where $j_{k}, \ldots, j_1, j_0, J$ are nonnegative integers satisfying
\[kj_{k} + (k - 1)j_{k-1} + \cdots + j_{1} + J = k,\]
\begin{equation}\label{n3}
  j_{k} + j_{k-1} + \cdots + j_{1} + j_0 = k.
\end{equation}

We claim that for $x\in \tilde{I},$
\begin{equation}\label{n2}
   |\eqref{n1}| \lesssim_k \frac{1}{\gamma(w,w')^k} |\partial_x^J G(x)|.
\end{equation}
In fact, for a nonnegative integer $m,$ we have
\[ \partial_x^m b (x)= \partial_x^m \left(\frac{1}{\phi'_w(x) - \phi_{w'}'(x)}\right)   \]
is a sum of $\mc{O}_m(1)$ many terms of the form $X(\phi'_{w}(x) - \phi'_{w'}(x))^{-y}(-1)^{y-1} (y-1)!,$ where $y$ is an integer satisfying $1\leq y \leq m+1$, and $X$ is a product of $(y-1)$ many terms in $\{\phi^{(l)}_{w'} - \phi^{(l)}_{w}: 2 \leq l \leq m+1\}.$ Because of Proposition \ref{phasesep} and Corollary \ref{kdercor} we have for $x\in \tilde{I},$
\[\| \partial^m_x b \|_{L^\infty (\tilde{I})} \lesssim_m
 \frac{( \gamma_{w_{\leq d(w,w')+1}})^{y-1} }{( \gamma_{w_{\leq d(w,w')+1}})^{y} } \sim \frac{1}{\gamma(w,w')}. \]
Recalling \eqref{n3}, we conclude \eqref{n2}.

Now we estimate $|\partial_x^{J} G(x)|.$ $\partial_x^{J} G(x)$ is a sum of $\mathcal{O}_k(1)$ many terms of the form
\begin{equation}\label{n5}
  \partial_x^{l_1}(|g_{w}'(x)|^{\mathrm{Re}(s)}{|g_{w'}'(x)|^{\mathrm{Re}(s)}})
  \partial_x^{l_2}(u(g_{w}(x))\overline{u(g_{w'}(x))})
   \partial_x^{l_3}( \chi_w (g_w(x)) \chi_{w'}(g_{w'}(x)))
\end{equation}
where $l_1,l_2,l_3$ are nonnegative integers satisfying $l_1+l_2+l_3 = J.$ Because of Proposition \ref{secondder}, we have for $x\in \tilde{I},$
\begin{equation*}
  |\partial_x^{l_1}(|g_{w}'(x)|^{\mathrm{Re}(s)}{|g_{w'}'(x)|^{\mathrm{Re}(s)}})| \lesssim_{l_1} |g_{w}'(x)|^{\mathrm{Re}(s)}{|g_{w'}'(x)|^{\mathrm{Re}(s)}},
\end{equation*}
and
\begin{equation*}
  |\partial_x^{l_2}(u(g_{w}(x))\overline{u(g_{w'}(x))}) | \lesssim_{l_2} \sum_{l_4 + l_5 \leq l_2} |g_{w}'(x)|^{l_4} |g_{w'}'(x)|^{l_5} |u^{(l_4)} (g_{w}(x))| |u^{(l_5)} (g_{w'}(x))|,
\end{equation*}
where $l_4,l_5$ are nonnegative integers. Because of Proposition \ref{secondder} and \eqref{n80}, we have for $x\in \tilde{I},$
\[|\partial^{l_3}( \chi_w (g_w(x)) \chi_{w'}(g_{w'}(x)))| \lesssim_{l_3}  1. \]
Since $|g_{w}'(x)|\sim |g_{w'}'(x)| \sim h^\rho$ for $x\in \tilde{I},$  combining the above three estimates we obtain for $x\in \tilde{I},$
\begin{equation}\label{n10}
  |\partial_x^{J}G(x)| \lesssim_{k} \sup_{0 \leq l_2 \leq J} \sum_{l_4 + l_5 \leq l_2} h^{\rho(l_4+l_5+2\mathrm{Re}(s)-1)} |g_{w}'(x)|^{1/2} |g_{w'}'(x)|^{1/2} |u^{(l_4)} (g_{w}(x))| |u^{(l_5)} (g_{w'}(x))|.
\end{equation}

Combining \eqref{n2} and \eqref{n10}, we have for $x\in \tilde{I},$
\begin{multline}\label{nnn2}
  |B_{w,w'}(x)| \lesssim_k  \frac{1}{ \gamma(w,w')^k  }
\sum_{l_2=0}^{J} \sum_{l_4 + l_5 \leq l_2} h^{\rho(l_4+l_5+2\mathrm{Re}(s)-1)} \\
   |g_{w}'(x)|^{1/2} |g_{w'}'(x)|^{1/2} |u^{(l_4)} (g_{w}(x))| |u^{(l_5)} (g_{w'}(x))|.
\end{multline}

\subsubsection{Estimate of $|A_{w,w'}|$}
As a consequence of \eqref{nnn2}, we have
\begin{multline*}
  |A_{w,w'}| \lesssim_k  \frac{h^{k}}{ \gamma(w,w')^k  }
\sum_{l_2=0}^{k} \sum_{l_4 + l_5 \leq l_2} h^{\rho(l_4+l_5+2\mathrm{Re}(s)-1)}
  \int_{\tilde{I}} |g_{w}'(x)|^{1/2} |g_{w'}'(x)|^{1/2} |u^{(l_4)} (g_{w}(x))| |u^{(l_5)} (g_{w'}(x))| dx.
\end{multline*}

Applying H\"{o}lder's inequality and the change of variable formula we obtain
\begin{equation}\label{n11}
  |A_{w,w'}| \lesssim_k \frac{h^k}{ \gamma(w,w')^k  } \sum_{l_2=0}^{k} \sum_{l_4 + l_5 \leq l_2} h^{\rho(l_4+l_5+2\mathrm{Re}(s)-1)} \|u^{(l_4)}\|_{L^2(I_w)} \|u^{(l_5)}\|_{L^2(I_{w'})}.
\end{equation}
By Proposition \ref{priori2} we conclude
\begin{equation}\label{n12}
  |A_{w,w'}| \lesssim_k \frac{h^k}{ \gamma(w,w')^k } h^{\rho k-k} h^{(2\mathrm{Re}(s)-1)\rho} \|u\|_{L^2(I_{\underline{w}})} \|u\|_{L^2(I_{\underline{w'}})} + \mathcal{O}_k(h^\infty) \|u\|_{L^2(I)}^2.
\end{equation}

\subsubsection{Estimate of $|T_j|$}
Finally, we sum up all $A_{w,w'}$ to get an estimate on $|T_j|.$ For a fixed $w \in Z(h^\rho),$ we have the trivial estimate
\begin{equation}\label{n20}
  |\{w' \in Z(h^\rho) : w \nsim w',\, \gamma(w,w') \in [2^j,2^{j+1})\}| \lesssim \left(\frac{h^\rho}{2^j}\right)^{-C_1},
\end{equation}
where $C_1>0$ is an admissible constant (depending only on $N,\gamma$ and $c$).
Consequently, by the Cauchy–Schwarz inequality we have
\[ \sum_{\substack{(w,w'): w, w' \in Z(h^\rho),\, w \nsim w' \\ \gamma(w,w') \in [2^j, 2^{j+1} )}} \|u\|_{L^2(I_{\underline{w}})} \|u\|_{L^2(I_{\underline{w'}})}
\lesssim \left(\frac{2^j}{h^\rho}\right)^{C_1} \sum_{w \in Z(h^\rho)} \|u\|_{L^2(I_{\underline{w}})}^2 \lesssim
\left(\frac{2^j}{h^\rho}\right)^{C_1} \|u\|_{L^2(I)}^2,
\]
where the last inequality is due to \eqref{n61}.
Therefore,
\begin{multline*}
  \sum_{\substack{(w,w'): w, w' \in Z(h^\rho),\, w \nsim w'\\ \gamma(w,w') \in [2^j,2^{j+1})}} |A_{w,w'}| \lesssim_k 2^{-jk} \left( \frac{2^j}{h^\rho} \right)^{C_1} h^{\rho k} h^{(2\mathrm{Re}(s)-1)\rho} \|u\|^2_{L^2(I)}
  + \left( \frac{1}{h^\rho} \right)^{C_1} \left( \frac{2^j}{h^\rho} \right)^{C_1} \mathcal{O}_k(h^\infty)  \|u\|_{L^2(I)}^2,
\end{multline*}
which implies
\begin{align}\label{nnn1}
  |T_j| & \lesssim_k  2^{-jk} \left( \frac{2^j}{h^\rho} \right)^{C_1} h^{\rho k} h^{(2\mathrm{Re}(s)-1)\rho} \|u\|^2_{L^2(I)} + h^{-2C_1\rho} 2^{C_1j} \mathcal{O}_k(h^\infty)  \|u\|_{L^2(I)}^2 \\
   & =  \left( \frac{h^\rho}{2^j} \right)^{k-C_1} h^{(2\mathrm{Re}(s)-1)\rho} \|u\|^2_{L^2(I)} + h^{-2C_1\rho} 2^{C_1j} \mathcal{O}_k(h^\infty)  \|u\|_{L^2(I)}^2. \nonumber
\end{align}

\subsection{Concluding the proof of Theorem \ref{maintheorem}}
We now choose and fix $k>C_1,$ so
\[\sum_{j=-M_1}^{M_2}  (h^\rho/2^j)^{k-C_1} \sim  (h^\rho/2^{-M_1})^{k-C_1} \sim 1,\]
where the last inequality is due to $2^{-M_1} \sim h^\rho.$
Also, we have
\[ \sum_{j=-M_1}^{M_2} h^{-2C_1\rho} 2^{C_1j} \mathcal{O}_k(h^\infty)  \|u\|_{L^2(I)}^2 \sim h^{-2C_1\rho} \mathcal{O}_k(h^\infty)  \|u\|_{L^2(I)}^2 = h^{-\mathcal{O}(\rho)} \mathcal{O} (h^\infty) \|u\|_{L^2(I)}^2 = \mathcal{O} (h^\infty) \|u\|_{L^2(I)}^2. \]

So combining \eqref{n30} and \eqref{nnn1} we conclude
\[\left|T + \sum_{j=-M_1}^{M_2} T_j \right| \leq |T| + \sum_{j=-M_1}^{M_2} |T_j| \lesssim h^{(2\mathrm{Re}(s)-1)\rho} \|u\|_{L^2(I)}^2 + \mathcal{O} (h^\infty) \|u\|_{L^2(I)}^2. \]
We recall \eqref{3}, and therefore
\begin{equation}\label{n50}
  \|u\|_{L^2(I_0)}^2 \lesssim h^{(2\mathrm{Re}(s)-1)\rho} \|u\|_{L^2(I)}^2 + \mathcal{O} (h^\infty) \|u\|_{L^2(I)}^2.
\end{equation}

Since $g_w: I \rightarrow I$ is eventually contracting to $\mathcal{J}$, and $I_0$ contains an open neighborhood of $\mathcal{J},$ there exists $M_0 \in \N$ sufficiently large depending on $I_0$ such that $g_w(I) \subset I_0$ for every $|w| = M_0.$ Therefore, from the identity $\mathcal{L}_s^{M_0} u = u$ on $I,$ we have
\[ \|u\|_{L^2(I)} = \|\mathcal{L}_s^{M_0} u\|_{L^2(I)} \lesssim \sum_{|w| = M_0}  \||g_{w}'(x)|^{s}u(g_{w}(x))\|_{L^2(I)}
\lesssim_{M_0} \|u\|_{L^2(I_0)}, \]
where the last inequality is by applying the change of variable formula.

Hence \eqref{n50} implies
\[ \|u\|_{L^2(I_0)}^2 \lesssim h^{(2\mathrm{Re}(s)-1)\rho} \|u\|_{L^2(I_0)}^2 + \mathcal{O} (h^\infty) \|u\|_{L^2(I_0)}^2. \]
Since $2\mathrm{Re}(s)-1 > 2\epsilon,$ when $h$ is sufficiently small depending on $\e,$ the above implies $\|u\|_{L^2(I_0)}^2 \leq \frac{1}{2} \|u\|_{L^2(I_0)}^2,$ and hence $\|u\|_{L^2(I_0)} = 0.$ Since $u\in \mathcal{H}$ is holomorphic, we conclude $u = 0.$ This proves Theorem \ref{maintheorem}.


Corollary \ref{maincoro} is an immediate consequence of Theorem \ref{maintheorem}, the fact that $Z(s)$ is entire, and the identity principle.

\section{Some Numerical Results} \label{numericsection}
We follow the method in \cite{jenkinson2002calculating}, in which they proposed the following approximations to the dynamical zeta function $Z(s):$
$$\Delta_N(s)=1+\sum_{n=1}^{N}\sum_{n_1+\cdots+n_m=n}\frac{(-1)^m}{m!}\prod_{l=1}^{m}
\frac{1}{n_l}\sum_{f^{n_l}z=z}|(f^{n_l})'(z)|^{-s}\left( 1+
\frac{1-2(f^{n_l})'(z)}{|(f^{n_l})'(z)|^2} \right)^{-1}.$$
Here the summation $\sum_{n_1+\cdots+n_m=n}$ is over all ordered sequences of positive integers $(n_1,\ldots, n_m)$ such that $n_1+\cdots+n_m=n.$
In \cite{jenkinson2002calculating}, the authors used such approximations to compute the Hausdorff dimension $\delta$ of the Julia set associated to $z^2+c$ (as well as the Hausdorff dimension of the Kleinian limit set), which corresponds to the largest real zero of $Z(s).$
The convergence turns out to be very fast near the point $s=\delta$ and for $s$ with small imaginary part.
Also, the smaller $\delta_c,$ the Hausdorff dimension of the Julia set associated with $z^2+c,$ is, the faster $\Delta_N$ converges. Note however that in our paper only the cases $\delta_c>1/2$ are of interest, so we need to compute $\Delta_N$ with sufficiently large $N$ in order to provide a good estimate for the zeta function. The correspondence between $c$ and $\delta_c$ is shown in figure \ref{fig:dimension}.

The method proposed by Jenkinson and Pollicott applies in a much more general context. See \cite{borthwick2014distribution, borthwick2016spectral} for computations of dynamical zeta functions for convex co-compact hyperbolic surfaces using this method. See also \cite{strain2004growth} for computations in the case of Julia sets as we have here, but using a different method.

\begin{figure}
\begin{subfigure}{.5\textwidth}
  \centering
  \includegraphics[width=1.05\linewidth]{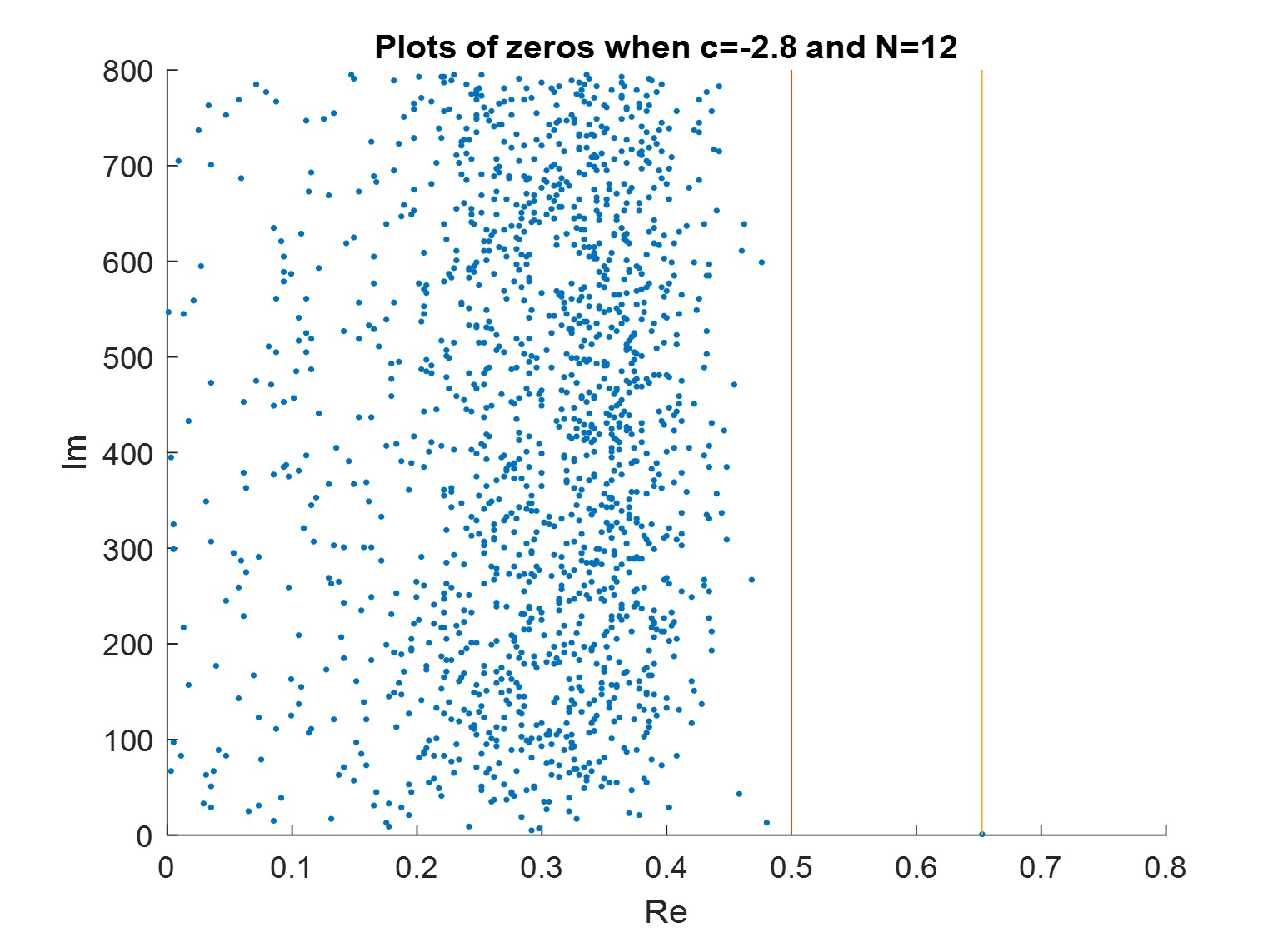}
\end{subfigure}%
\begin{subfigure}{.5\textwidth}
  \centering
  \includegraphics[width=\linewidth]{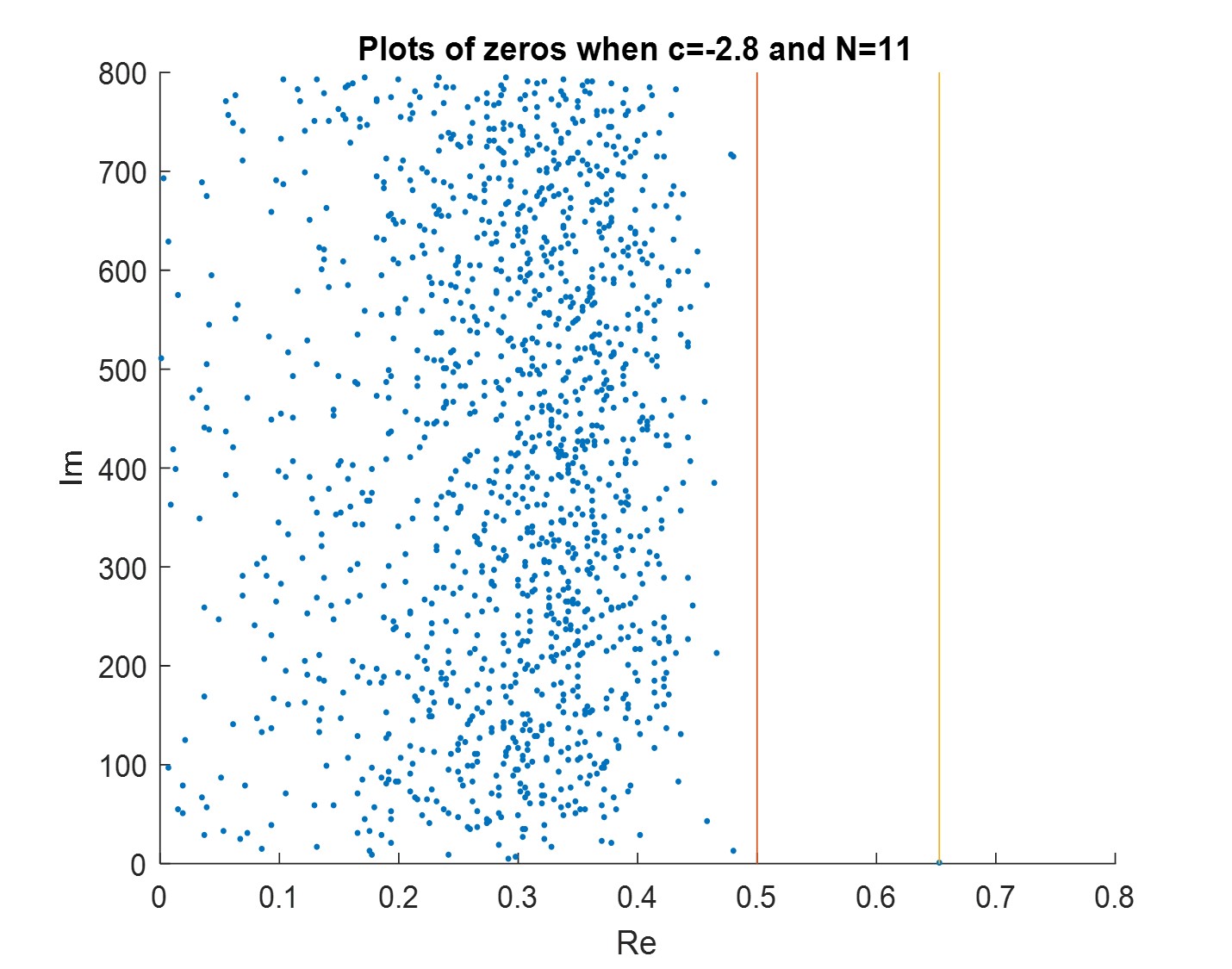}
\end{subfigure}
\caption{Plots when $c=-2.8.$ The two vertical lines in the figures are $\mathrm{Re}(s)=1/2$ and $\mathrm{Re}(s)=\delta$ respectively. Dots represent zeros of $\Delta_N.$}
\label{fig:28}
\end{figure}

\begin{figure}
\begin{subfigure}{.5\textwidth}
  \centering
  \includegraphics[width=1.04\linewidth]{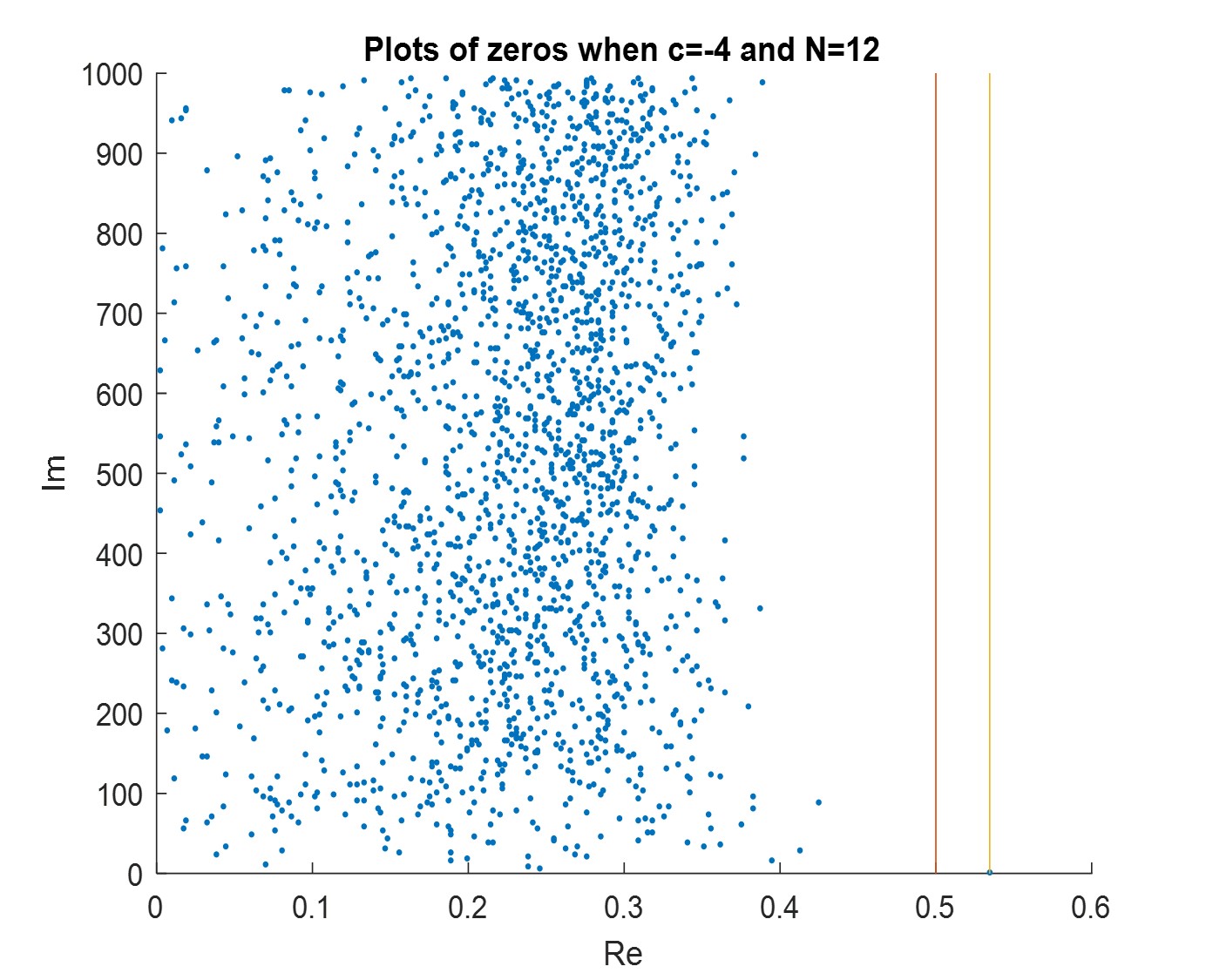}
\end{subfigure}%
\begin{subfigure}{.5\textwidth}
  \centering
  \includegraphics[width=\linewidth]{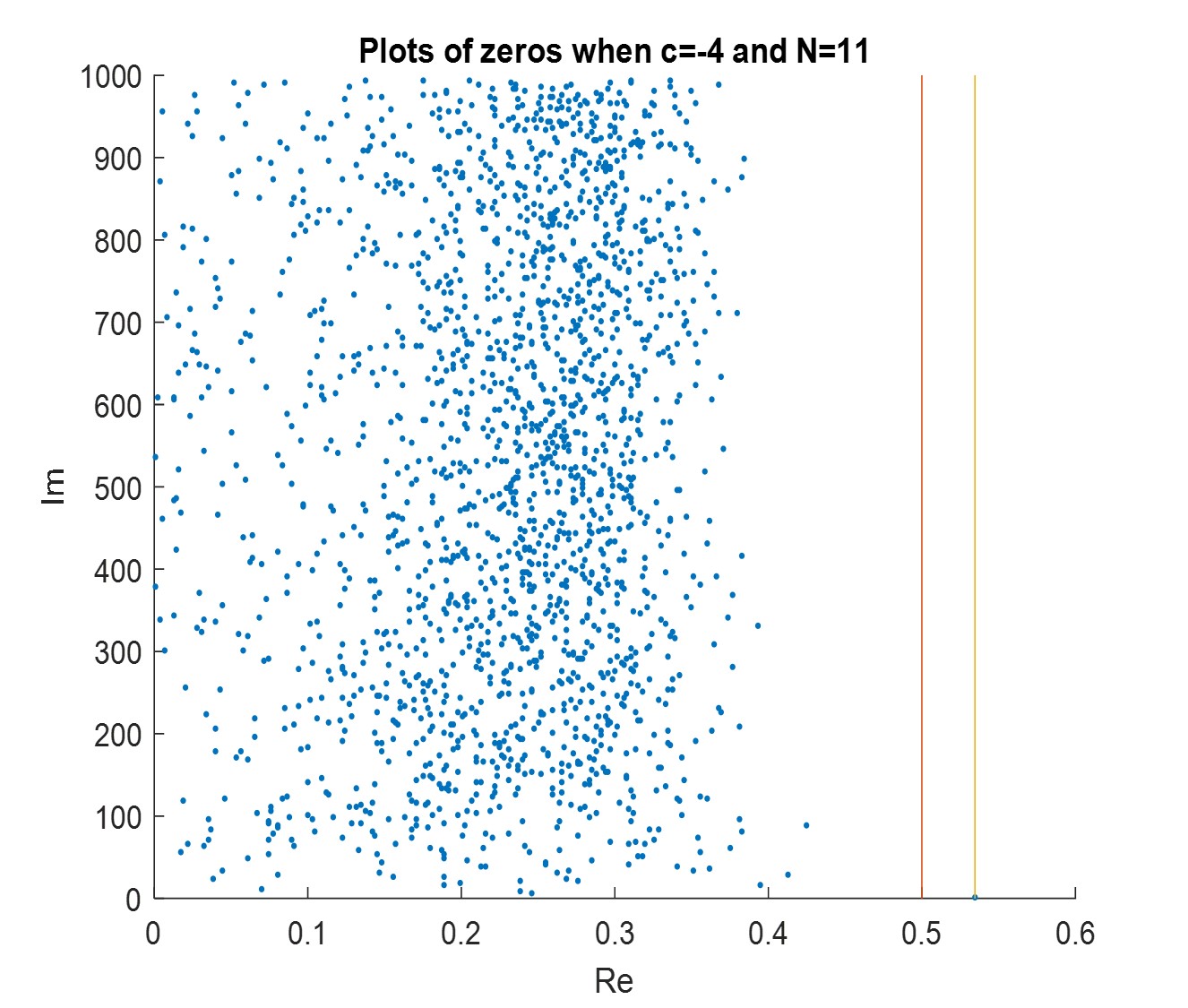}
\end{subfigure}
\caption{Plots when $c=-4.$ The two vertical lines in the figures are $\mathrm{Re}(s)=1/2$ and $\mathrm{Re}(s)=\delta$ respectively. Dots represent zeros of $\Delta_N.$}
\label{fig:4}
\end{figure}

Here we present the plots of zeros of $Z(s)$ for $c=-2.8$ and $c=-4.$ Limited by the computing time, we only compute approximations up to $N=12.$ It seems that $N=12$ could not provide an accurate pattern of the distribution of zeros (except for the region near $s=\delta$ or near the real axis), due to the slow convergence when $\delta$ is large.
However, a trend is that as $N$ grows larger, zeros tend to move leftwards. Therefore the two approximated plots (Figure \ref{fig:28} and \ref{fig:4}) we present here could be instructive in telling us whether the essential zero-free strips exist. In fact, these two figures show empirically that the essential zero-free strips $\mathrm{Re}(s)>1/2+\e$ do exist, or even zero-free strips exist. So we conjecture that essential zero-free strips $\mathrm{Re}(s)>1/2+\e$ exist for $c$ in some range in $(-\infty,-2)$ much larger than $(-\infty,-3.75)$ as we have proved here.

\section*{Acknowledgements}
I would like to thank Prof. Semyon Dyatlov for suggesting this problem and many helpful discussions. I am also thankful to Prof. Maciej Zworski for introducing me to this topic. Partial support from NSF CAREER grant DMS-1749858 is acknowledged.

\bibliographystyle{alpha}
\bibliography{Zeta}

\end{document}